\documentclass[reqno,11pt]{amsart}
\usepackage{amsmath,amssymb,amsthm,graphicx,mathrsfs,bbm,url}
\usepackage{amsthm}
\usepackage{wrapfig}
\usepackage{enumitem}
\usepackage{mathtools}
\usepackage[utf8]{inputenc}
 \usepackage[T1]{fontenc}
\usepackage[usenames,dvipsnames]{color}
\usepackage[colorlinks=true,linkcolor=Red,citecolor=Green]{hyperref}
\usepackage[super]{nth}
\usepackage[open, openlevel=2, depth=3, atend]{bookmark}
\hypersetup{pdfstartview=XYZ}
\usepackage[font=footnotesize]{caption}
\usepackage{a4wide}

\usepackage{epstopdf}
 
\usepackage{hyperref}

\theoremstyle{plain}
\newtheorem{theorem}{Theorem}[section]
\newtheorem*{theorem*}{Theorem}

\newtheorem{lemma}[theorem]{Lemma}
\newtheorem{proposition}[theorem]{Proposition}

\theoremstyle{definition}
\newtheorem{definition}[theorem]{Definition}
\newtheorem{example}[theorem]{Example}

\theoremstyle{remark}
\newtheorem{remark}[theorem]{Remark}

\numberwithin{equation}{section}

\newcommand{\C}{\mathbb{C}}
\newcommand{\R}{\mathbb{R}}

\newcommand{\M}{\mathcal{M}}
\newcommand{\N}{\mathbb{N}}

\newcommand{\X}{\mathbf{X}}

\newcommand{\Ss}{\mathbb{S}}
\newcommand{\eps}{\varepsilon}

\newcommand{\mc}{\mathcal}

\newcommand{\dd}{\mathrm{d}}
\newcommand{\e}{\mathbf{e}}

\DeclareMathOperator{\vol}{vol}

\DeclareMathOperator{\Tr}{Tr}

\DeclareMathOperator{\WF}{WF}

\DeclareMathOperator{\ran}{ran}
\DeclareMathOperator{\id}{id}

\DeclareMathOperator{\Id}{id}
\DeclareMathOperator{\supp}{supp}

\DeclareMathOperator{\rk}{rank}
\DeclareMathOperator{\ind}{ind}

\DeclareMathOperator{\E}{\mathcal{E}}
\DeclareMathOperator{\F}{\mathcal{F}}
\DeclareMathOperator{\End}{End}
\DeclareMathOperator{\sk}{sk}
\DeclareMathOperator{\RR}{\mathbf{R}}

\newcommand{\be}{\begin{equation}}
\newcommand{\ee}{\end{equation}}

% ...

% ...

\title
[Generic dynamical properties of connections on vector bundles]
{Generic dynamical properties of connections on vector bundles}

\author[M. Ceki\'{c}]{Mihajlo Ceki\'{c}}
\date{\today}
\address{Laboratoire de Math\'{e}matiques d'Orsay, Univ. Paris-Sud, CNRS, Universit\'{e} Paris-Saclay, 91405 Orsay, France}
\email{mihajlo.cekic@u-psud.fr}

\author[T. Lefeuvre]{Thibault Lefeuvre}

\address{Laboratoire de Mathématiques d’Orsay, Univ. Paris-Sud, CNRS, Université Paris-Saclay, 91405 Orsay, France}

\email{thibault.lefeuvre@u-psud.fr}

\begin{document}

\begin{abstract}
Given a smooth Hermitian vector bundle $\mc{E}$ over a closed Riemannian manifold $(M,g)$, we study generic properties of unitary connections $\nabla^{\mc{E}}$ on the vector bundle $\mc{E}$. First of all, we show that twisted Conformal Killing Tensors (CKTs) are generically trivial when $\dim(M) \geq 3$, answering an open question of Guillarmou-Paternain-Salo-Uhlmann \cite{Guillarmou-Paternain-Salo-Uhlmann-16}. In negative curvature, it is known that the existence of twisted CKTs is the only obstruction to solving exactly the twisted cohomological equations which may appear in various geometric problems such as the study of transparent connections. The main result of this paper says that these equations can be generically solved. As a by-product, we also obtain that the induced connection $\nabla^{\mathrm{End}(\E)}$ on the endomorphism bundle $\mathrm{End}(\E)$ has generically trivial CKTs as long as $(M,g)$ has no nontrivial CKTs on its trivial line bundle. Eventually, we show that, under the additional assumption that $(M,g)$ is Anosov (i.e. the geodesic flow is Anosov on the unit tangent bundle), the connections are generically \emph{opaque}, namely there are no non-trivial subbundles of $\mc{E}$ which are generically preserved by parallel transport along geodesics. The proofs rely on the introduction of a new microlocal property for (pseudo)differential operators called \emph{operators of uniform divergence type}, and on perturbative arguments from spectral theory (especially on the theory of Pollicott-Ruelle resonances in the Anosov case).
\end{abstract}

\maketitle

\section{Introduction}

\subsection{Generic absence of twisted Conformal Killing Tensors}

\subsubsection{Statement of the main result.} 

Consider $(\mc{E},\nabla^{\mc{E}})$ a Hermitian vector bundle equipped with a unitary connection over a smooth Riemannian $n$-manifold $(M,g)$ with $n \geq 2$. Let $SM$ be the unit sphere bundle and $\pi : SM \rightarrow M$ be the projection. We consider the pullback bundle $(\pi^* \mc{E}, \pi^* \nabla^{\mc{E}})$ over $SM$ and we will forget the $\pi^*$ in the following in order to simplify the notations. Denote by $X$ the geodesic vector field and define the operator $\X := \nabla^{\mc{E}}_X$, acting on sections of $C^\infty(SM,\mc{E})$. By standard Fourier analysis, we can write $f \in C^\infty(SM,\mc{E})$ as $f = \sum_{m \geq 0} f_m$, where $f_m \in C^\infty(M, \Omega_m \otimes \mc{E})$ and pointwise in $x \in M$:
\[
\Omega_m \otimes \mc{E} := \ker(\Delta^{\mc{E}}_v + m(m+n-2)),
\]
is the kernel of the vertical Laplacian (note that this Laplacian is independent of the connection $\nabla^{\mc{E}}$, it only depends on $\mc{E}$ and on $g$). Elements in this kernel are called the \emph{twisted spherical harmonics of degree $m$}. We will say that $f \in C^\infty(SM,\mc{E})$ has \emph{finite Fourier content} if its expansion in spherical harmonics only contains a finite number of terms. The operator $\X$ maps
\begin{equation}\label{eq:XX}
\X : C^\infty(M, \Omega_m \otimes \mc{E}) \rightarrow C^\infty(M, \Omega_{m-1} \otimes \mc{E}) \oplus C^\infty(M, \Omega_{m+1} \otimes \mc{E})
\end{equation}
and can be decomposed as $\X = \X_+ + \X_-$, where, if $u \in C^\infty(M,\Omega_m \otimes \mc{E})$, $\X_+u \in C^\infty(M,\Omega_{m+1} \otimes \mc{E})$ denotes the orthogonal projection on the twisted spherical harmonics of degree $m+1$. The operator $\X_+$ is elliptic and thus has finite-dimensional kernel whereas $\X_-$ is of divergence type (see \S\ref{section:microlocal} for definitions). Moreover, $\X_+^* = -\X_-$, where the adjoint is computed with respect to the canonical $L^2$ scalar product on $SM$ induced by the Sasaki metric. We refer to the original articles of Guillemin-Kazhdan \cite{Guillemin-Kazhdan-80, Guillemin-Kazhdan-80-2} for a description of these facts and to \cite{Guillarmou-Paternain-Salo-Uhlmann-16} for a more modern exposition.

\begin{definition}
We call \emph{twisted Conformal Killing Tensors} (CKTs) elements in the kernel of $\X_+|_{C^\infty(M,\Omega_m \otimes \mc{E})}$. 
\end{definition}

We say that the twisted CKTs are trivial when the kernel is reduced to $\left\{0\right\}$. The goal of this paper is to investigate the generic properties of non trivial twisted CKTs. It is known that the kernel of $\X_+$ is invariant by a conformal change of the metric $g$ (see \cite[Section 3.6]{Guillarmou-Paternain-Salo-Uhlmann-16}). Moreover, in negative curvature, it is known that there exists $m_0 \geq 0$ such that for all $m \geq m_0$, $\ker \X_+|_{C^\infty(M,\Omega_m \otimes \mc{E})}$ is trivial (see \cite[Theorem 4.5]{Guillarmou-Paternain-Salo-Uhlmann-16}). The number $m_0$ can be estimated explicitly in terms of an upper bound of the sectional curvature of $(M,g)$ and the curvature of the vector bundle $\mc{E}$. This should still be true for Anosov manifolds (namely, manifolds whose geodesic flow is uniformly hyperbolic on the unit tangent bundle) although it is still an open question for the moment. In particular, this is true for Anosov surfaces since any Anosov surface $(M,g)$ is conformally equivalent to a metric of negative curvature and twisted CKTs are invariant by conformal changes. Note that there are examples of CKTs on any manifold of any dimension $n \geq 2$: simply consider the vector bundle $(TM,\nabla_{\mathrm{LC}}) \rightarrow M$ equipped with the Levi-Civita connection; then the ``tautological section'' $s \in C^\infty(SM,\pi^*TM)$ defined by $s(x,v) := v$ satisfies $\X_+ s = 0$. Nevertheless, it was conjectured in \cite[Section 1]{Guillarmou-Paternain-Salo-Uhlmann-16} that the absence of twisted CKTs should be a generic property. In this paper, we answer positively to this question:

\begin{theorem}
\label{theorem:main}
Let $\mc{E}$ be a Hermitian vector bundle over a smooth Riemannian manifold $(M,g)$ with $\dim M \geq 3$\footnote{In the case $\dim M = 2$, the operators $\mathbf{X}_{\pm}$ are elliptic which changes the problem. For our results in this case, see Proposition \ref{prop:n=2}.}. Then, for any $k \geq 2$, there is a residual set of unitary connections with regularity $C^k$ such that $(\mc{E},\nabla^{\mc{E}})$ has no CKTs.
\end{theorem}

Observe that this result holds without any assumption on the curvature. For fixed $m \in \N$, absence of twisted CKTs is an open property: this is a mere consequence of standard elliptic theory for (pseudo)differential operators. As a consequence, in order to prove that connections without CKTs form a residual set (i.e. a countable intersection of open and dense sets), it is sufficient to show that for fixed $m \in \N$, unitary connections without twisted CKTs of degree $m$ form a dense set of unitary connections with regularity $C^k$: we prove this property in Theorem \ref{theorem:ckts}. As mentioned earlier, if $(M,g)$ has negative curvature, there exists a maximal $m_0 \in \N$, such that for any $m \geq m_0$, one knows that there are no CKTs of degree $m$ and thus the set of connections with no nontrivial CKTs residual is also open. 

In order to prove that unitary connections without twisted CKTs of degree $m$ form a dense set of unitary connections with regularity $C^k$, we use a perturbative argument for the eigenvalues of a certain natural Laplacian operator $\Delta_+ = (\X_+)^*\X_+ \geq 0$ acting on $C^\infty(M,\Omega_m \otimes \mc{E})$ (see \S \ref{ssection:perturbation}) whose kernel is given by the twisted CKTs of degree $m$. Due to the affine structure of the set of connections, it is possible to compute explicitly the first and second variation (with respect to the connection) of the $0$ eigenvalue of the operator $\Delta_+$. The first variation vanishes but the second variation\footnote{We also point out that we initially tried to apply the perturbation theory of Uhlenbeck \cite[Theorem 1]{Uhlenbeck-76} but we could not make it work. The main reason, it seems, is that we need to look at a second order variation of the eigenvalues at $0$ whereas the transversality theory in \cite{Uhlenbeck-76} is a ``first-order perturbation'' theory.} is given by some explicit non-negative quantity (see Lemma \ref{lemma:second-variation}). It is then sufficient to produce a perturbation of the connection such that this second variation is strictly positive. 

\subsubsection{Operators of uniform divergence type.}

In order to do so, we introduce a new notion of microlocal analysis which we call \emph{operators of uniform divergence type} (see \S\ref{section:microlocal}). This is the main input of the present paper and could be of independent interest. These (pseudo)differential operators are of the form $Q : C^\infty(M,F) \rightarrow C^\infty(M,E)$ where $\rk(F) > \rk(E)$. They are of \emph{divergence type} in the sense that the principal symbol of the adjoint is injective for all $(x,\xi) \in T^*M \setminus \left\{ 0 \right\}$, i.e. $\sigma_{Q^*}(x,\xi) \in \mathrm{Hom}(E_x,F_x)$ is injective (equivalently the principal symbol $\sigma_Q(x, \xi) \in \mathrm{Hom}(F_x, E_x)$ is surjective). The uniform divergence type property asserts that
\[
\Sigma_{|\xi|=1} \ker \sigma_Q(x,\xi) = F_x,
\]
for all $x \in M$ and allows to describe the \emph{values} that elements in $\ker Q|_{C^\infty(M,F)}$ can take at a given point $x$. In particular, under this property, the map
\[
\mathrm{ev}_x : \ker Q|_{C^\infty(M,F)} \rightarrow F_x, ~~~ \mathrm{ev}_x(f) := f(x)
\]
is surjective for all $x \in M$. 

We also point out that the perturbation used is a priori \emph{global} on $M$ but it could be interesting to see whether our result can be made local in the following sense: given an open subset $\Omega \subset M$, the equation $\X_+ u = 0$ on $\Omega$ has generically (with respect to the connection) only trivial solutions. Such a local perturbative argument is developed in \cite{Kruglikov-Matveev-16} who show that generically a metric has no \emph{Killing tensors} (see \S\ref{ssection:twisted-tensors} for a definition). This would require extra work on operators of uniform divergence type and we plan to investigate this in the future. More generally, the method developed in the present article seems fairly robust in order to deal with general linear perturbations of gradient-type or Laplacian-type operators.


Finally, let us briefly mention that in negative curvature, twisted CKTs are an obstruction to solving exactly some transport equations called \emph{twisted cohomological equations} which appear in some geometric settings such as the study of transparent pairs of connections. 
Assume $(M, g)$ is negatively curved and denote $r := \rk \E$. A closed geodesic on $M$ can be identified with a periodic orbit $\gamma$ for the geodesic flow on $SM$, and one can look at the holonomy induced by the pullback connection along $\gamma$, i.e. the parallel transport of sections of $\pi^*\mc{E}$ along the geodesic lines. We say that a connection is \emph{transparent} if the holonomy is trivial along all periodic orbits of the geodesic flow (see \cite{Paternain-09,Paternain-11,Paternain-12,Paternain-13,Guillarmou-Paternain-Salo-Uhlmann-16,Cekic-Lefeuvre-20} for the study of this question). In this case, it is known that $\pi^*\mc{E}$ is trivial over $SM$ (see \cite{Cekic-Lefeuvre-20} for instance) and one can prove that there exists a smooth family $(e_1,...,e_r) \in C^\infty(SM,\pi^* \mc{E})$ which is independent at every point $(x,v) \in SM$ (it trivializes the bundle $\pi^*\mc{E}$ over $SM$) and such that $\pi^*\nabla^{\mc{E}}_X e_i = 0$ for $i=1,...,r$. We call such an equality/equation a \emph{twisted cohomological equation}: it is a transport equation on the unit tangent bundle involving some vector bundle. More generally, twisted cohomological equations are of the form $\pi^*\nabla^{\mc{E}}_X u = f$, where $f = f_0 + ... + f_N$ and for $0 \leq i \leq N$, $f_i \in C^\infty(M,\Omega_i \otimes \E)$, i.e. $f$ has finite Fourier content. In negative curvature, it is known that such a twisted cohomological equation imply that the sections $e_i \in C^\infty(SM,\pi^* \mc{E})$ have finite Fourier content (see \cite[Theorem 4.1]{Guillarmou-Paternain-Salo-Uhlmann-16}). If one can prove that the $e_i$ are actually \emph{independent of the velocity variable} i.e. $e_i \in C^\infty(M,\Omega_0 \otimes \mc{E}) \simeq C^\infty(M,\mc{E})$, then this implies that they are actually sections living on the base manifold $M$ and the equation $\pi^* \nabla^{\mc{E}}_X e_i = 0$ is equivalent to $\nabla^{\mc{E}} e_i = 0$, i.e. these sections are parallel. In other words the vector bundle $(\mc{E},\nabla^{\mc{E}})$ over $M$ is isomorphic to the trivial bundle $(\C^r,d)$ equipped with the trivial flat connection. In order to prove that the $e_i$ are indeed independent of the velocity variable, it is sufficient to know that the connection $\nabla^{\mc{E}}$ has no non trivial twisted CKTs (see \cite[Theorem 5.1]{Guillarmou-Paternain-Salo-Uhlmann-16}), hence the importance of their study. The existence/non-existence of CKTs can also be investigated on manifolds with boundary: it is proved in \cite{Dairbekov-Sharafutdinov-10,Guillarmou-Paternain-Salo-Uhlmann-16} that there are no (twisted) CKTs which identically vanish on the boundary or on a hypersurface. 

\subsection{Generic absence of CKTs on the endomorphism bundle} 

There is a more general question than that of uniqueness for transparent connections. Indeed, one can ask the following inverse problem: does the holonomy of the connection along closed geodesics \emph{stably} determine the connection? We refer to \cite{Cekic-Lefeuvre-20} for an extensive discussion of this question: this is intimately related to the existence of non-trivial CKTs for the induced connection on the endomorphism bundle $\mathrm{End}(\E)$. Recall that a unitary connection $\nabla^{\mc{E}}$ on the Hermitian vector bundle $\mc{E} \rightarrow M$, induces a canonical connection $\nabla^{\mathrm{End}(\E)}$ on the endomorphism bundle (see \S\ref{ssection:absence-ckts-endomorphism} for a definition). We can also investigate the existence/absence of CKTs for this particular type of connection $\nabla^{\mathrm{End}(\E)}$ on $\mathrm{End}(\E)$. It is straightforward to check that $\nabla^{\mathrm{End}(\E)} \mathbbm{1}_{\mc{E}} = 0$ that is, there is always a CKT of order $m=0$. We will say that $\nabla^{\mathrm{End}(\E)}$ has only trivial CKTs if $\mathbbm{1}_{\mc{E}}$ is the only CKT. We show the following:

\begin{theorem}
\label{theorem:generic-ckts-endomorphism}
Let $\mc{E}$ be a Hermitian vector bundle over a smooth Riemannian manifold $(M,g)$. Assume $(M,g)$ has no CKTs on its trivial line bundle. Then, for any $k \geq 2$, there is a residual set of unitary connections with regularity $C^k$ such that $(\mathrm{End}(\mc{E}),\nabla^{\mathrm{End}(\E)})$ has no non-trivial CKTs.
\end{theorem}

It is not clear whether one can drop the assumption that $(M,g)$ has no nontrivial CKTs on its trivial line bundle. This is known to hold in a quite general context (see \cite[Corollary 3.6]{Paternain-Salo-Uhlmann-15}):
\begin{itemize}
\item If $(M,g)$ has negative curvature and more generally if $g$ is in the conformal class of a negatively-curved metric since CKTs are invariant by conformal changes,
\item If $(M,g)$ has non-positive curvature and transitive geodesic flow and more generally if $g$ is in the conformal class of such a metric.
\end{itemize}
It is still not known (although expected) whether Anosov Riemannian manifolds, namely Riemannian manifolds whose geodesic flow is Anosov on the unit tangent bundle (see \eqref{equation:anosov} for a definition), have no nontrivial CKTs on their trivial line bundle. This could follow from the conjectural statement that any Anosov Riemannian manifold lies in the conformal class of a negatively-curved metric. Once again, this is not known unless $\dim(M)=2$.

Observe that the main difference between the previous Theorem \ref{theorem:generic-ckts-endomorphism} and Theorem \ref{theorem:main} is that the CKT generated by $\mathbbm{1}_{\mc{E}}$ for $m=0$ can never be excluded. As in Theorem \ref{theorem:main}, it is sufficient to prove that for fixed $m$, one can perturb the connection by $\nabla^{\E} + \Gamma$ so that this new connection has no nontrivial CKTs of degree $m$ and the proof relies on the same principle of operators of uniform divergence type. Also, in negative curvature, this set is not only residual, but it is also open just as in Theorem \ref{theorem:main}.

\subsection{Generic opacity of connections} Let $(\varphi_t)_{t \in \R}$ be the geodesic flow generated by the geodesic vector field $X$ on $SM$. We study the parallel transport along geodesics on the pullback bundle $\pi^* \mc{E} \rightarrow SM$. For $t \in \R, (x,v) \in SM$, we denote by $C((x,v),t) : \mc{E}_x \rightarrow \mc{E}_{\pi(\varphi_t(x,v))}$ the parallel transport of sections along the geodesic segment $(\pi\varphi_s(x,v))_{s \in [0,t]}$.

\begin{definition}[Invariant subbundles]
Let $\mc{F} \rightarrow M$ be a smooth subbundle of $\mc{E} \rightarrow M$. We say that $\mc{F}$ is \emph{invariant} if the following holds: for all $(x,v) \in SM, f \in \mc{F}_x$, one has $C((x,v),t) f \in \mc{F}_{\pi(\varphi_t(x,v))}$.  
\end{definition}

There is actually a more general notion of invariant subbundles defined on an arbitrary manifold $\M$ carrying a flow (see \S\ref{ssection:remarks-endomorphisms}, in other words, one does not need to take $\M = SM$). We introduce the following terminology for connections without invariant subbundles:

\begin{definition}[Opaque connections]
We say that the connection $\nabla^{\mc{E}}$ on $\mc{E}$ is \emph{opaque} if any invariant subbundle is trivial, i.e. is either $\mc{E}$ or $\left\{0\right\}$.
\end{definition}

We could have also chosen the terminology \emph{irreducible connection} but we chose \emph{opaque} instead to contrast with the notion of \emph{transparent connection}.
From now on, we will assume that the geodesic flow of $(M,g)$ is Anosov (see \eqref{equation:anosov} for a definition). Using the well-studied theory of Pollicott-Ruelle resonances (see \cite{Liverani-04, Gouezel-Liverani-06,Butterley-Liverani-07,Faure-Roy-Sjostrand-08,Faure-Sjostrand-11,Faure-Tsuji-13,Dyatlov-Zworski-16} for further details), it is possible to define a spectral theory for the (non-elliptic) first order differential operator
\[
(\pi^* \nabla^{\mathrm{End}(\E)})_X : C^\infty(SM, \mathrm{End}(\pi^*\E)) \rightarrow C^\infty(SM, \mathrm{End}(\pi^*\E)).
\]
(We added the $\pi^*$ here to insist on the fact that all the objects are obtained as pullbacks of objects defined downstairs on $M$.) We prove in Lemma \ref{lemma:observations} that the existence of invariant subbundles is equivalent to the existence of a Pollicott-Ruelle resonance at $\lambda = 0$ for the operator $(\pi^* \nabla^{\mathrm{End}(\E)})_X$ (with smooth resonant states). As before, the constant section $\mathbbm{1}_{\mc{E}}$ is always a trivial resonant state, namely $(\pi^* \nabla^{\mathrm{End}(\E)})_X \mathbbm{1}_{\mc{E}} = 0$. Using arguments from spectral theory, we show that we can perturb the connection $\nabla^{\mc{E}}$ by $\nabla^{\mc{E}} + \Gamma$ (for an arbitrary small $\Gamma$) and eject all the resonant states at $0$, except $\mathbbm{1}_{\mc{E}}$. This gives the following Theorem:

\begin{theorem}
\label{theorem:generic-opacity}
Let $\mc{E}$ be a Hermitian vector bundle over a smooth Anosov Riemannian manifold $(M,g)$. Then, there exists $k_0 \geq 0$ such that for any $k \geq k_0$, there is an open dense subset of unitary $C^k$ connections on $\mc{E}$ which are opaque.
\end{theorem} 

Let us make some important remarks. In the particular case where $(M,g)$ has negative curvature, Theorem \ref{theorem:generic-opacity} is a straightforward consequence of Theorem \ref{theorem:generic-ckts-endomorphism}. Indeed, by Theorem \ref{theorem:generic-ckts-endomorphism}, a connection has generically no nontrivial CKTs for $\nabla^{\mathrm{End}(\E)}$. Moreover, it is known that in negative curvature the equation $(\pi^* \nabla^{\mathrm{End}(\E)})_X u = 0$ (it is called a \emph{twisted cohomological equation}) implies that $u$ has finite Fourier content (see \cite[Theorem 4.1]{Guillarmou-Paternain-Salo-Uhlmann-16}) and the absence of non-trivial CKTs forces $u$ to be equal to $c \mathbbm{1}_{\E}$ for some $c \in \C$ (see \cite[Theorem 5.1]{Guillarmou-Paternain-Salo-Uhlmann-16}). In the general case of an Anosov manifold, it is conjectured that the same should happen, namely a solution to $(\pi^* \nabla^{\mathrm{End}(\E)})_X u = 0$ should have finite Fourier content, but this is still out of reach of the existing techniques.

In Theorem \ref{theorem:generic-opacity}, the minimal regularity $k_0 \geq 0$ needs to be increased because the arguments rely on microlocal analysis, and it is not clear whether the tools can work in low regularity. Using the fact that Pollicott-Ruelle resonances depend continuously on the operator (see \cite{Bonthonneau-19} for instance), one obtains that the set of opaque connections is open. It is therefore sufficient to prove that it is dense. \\

\noindent \textbf{Acknowledgement:} M.C. and T.L. have received funding from the European Research Council (ERC) under the European Union’s Horizon 2020 research and innovation programme (grant agreement No. 725967).

\section{Algebraic preliminaries}

\subsection{Symmetric tensors: definitions and properties}

\label{section:tensors}

We recall some elementary properties of symmetric tensors on Riemannian manifolds. The reader is referred to \cite{Dairbekov-Sharafutdinov-10} for an extensive discussion.

\subsubsection{Symmetric tensors in Euclidean space}

We consider a $n$-dimensional Euclidean vector space $(E,g_E)$ with orthonormal frame $(\e_1, ..., \e_n)$. We denote by $\otimes^m E^*$ the $m$-th tensor power of $E^*$ and by $\otimes^m_S E^*$ the symmetric tensors of order $m$, namely the tensors $u \in \otimes^m E^*$ satisfying:
\[
u(v_1,...,v_m) = u(v_{\sigma(1)},...,v_{\sigma(m)}),
\]
for all $v_1, ...,v_m \in E$ and $\sigma \in \mathfrak{S}_m$, the permutation group of $\left\{1, ..., m \right\}$. If $K = (k_1, ..., k_m) \in \left\{1, ...,n\right\}^m$, we define $\e_K^* = \e_{k_1}^* \otimes ... \otimes \e_{k_m}^*$, where $\e_i^*(\e_j) := \delta_{ij}$. We introduce the symmetrization operator $\mc{S} : \otimes^m E^* \rightarrow \otimes^m_S E^*$ defined by:
\[
\mc{S}(\eta_1 \otimes ... \otimes \eta_m) := \dfrac{1}{m!} \sum_{\sigma \in \mathfrak{S}_m} \eta_{\sigma(1)} \otimes ... \otimes \eta_{\sigma(m)},
\]
where $\eta_1, ...,\eta_m \in E^*$. Given $v \in E$, we define $v^\flat \in E^*$ by $v^\flat(w) := g_E(v,w)$ and call $\flat : E \rightarrow E^*$ the musical isomorphism, following the usual terminology. Its inverse is denoted by $\sharp : E^* \rightarrow E$. The scalar product $g_E$ naturally extends to $\otimes^m E^*$ (and thus to $\otimes^m_S E^*$) using the following formula:
\[
g_{\otimes^m E^*}(v_1^\flat \otimes ... \otimes v_m^\flat, w_1^\flat \otimes ... \otimes w_m^\flat) := \prod_{j=1}^m g_E(v_j,w_j),
\] 
where $v_i,w_i \in E$. In particular, if $u = \sum_{i_1, ..., i_m=1}^n u_{i_1 ... i_m} \e_{i_1}^* \otimes ... \otimes \e_{i_m}^*$, then $\|u\|_{\otimes^m E^*}^2 = \sum_{i_1,...,i_m=1}^n |u_{i_1 ... i_m}|^2$. For the sake of simplicity, we will still write $g_E$ instead of $g_{\otimes^m E^*}$. The operator $\mc{S}$ is an orthogonal projection with respect to this scalar product. 

There is a natural trace operator $\mc{T} : \otimes^m E^* \rightarrow \otimes^{m-2} E^*$ (it is formally defined to be $0$ for $m=0,1$) given by:
\begin{equation}
\label{equation:trace-tensor}
\mc{T}u := \sum_{i=1}^n u(\e_i,\e_i, \cdot, ...,\cdot),
\end{equation}
and it also maps $\mc{T} : \otimes^m_S E^* \rightarrow \otimes^{m-2}_S E^*$. Its adjoint (with respect to the metric $g_{\otimes^m E^*}$) on symmetric tensors is the map $\mc{J} : \otimes^m_S E^* \rightarrow \otimes^{m+2}_S E^*$ given by $\mc{J}u := \mc{S}(g_E \otimes u)$. It is easy to check that the map $\mc{J}$ is injective. This implies by standard linear algebra that one has the decomposition, where $\otimes_S^m E^*|_{0-\Tr} = \ker \mc{T} \cap \otimes_S^m E^*$ denotes the trace-free symmetric $m$-tensors:
\begin{equation}\label{eq:symetrictensorsplitting}
\otimes^m_S E^* = \otimes^m_S E^*|_{0-\Tr} \oplus^\bot \mc{J} \otimes^{m-2}_S E^* = \oplus_{k \geq 0} \mc{J}^k\otimes^{m-2k}_S E^*|_{0-\Tr}.
\end{equation}

Given $K = (k_1,...,k_m) \in \left\{1,...,n\right\}^m$, we introduce $\Theta(K) := (\theta_1(K),...,\theta_n(K))$, where $\theta_i(K) = \sharp\left\{k_j = i ~|~j =1,...,m\right\}$. Observe that $\mc{S} \e_K^* = \mc{S} \e_{K'}^*$ if and only if $\Theta(K)=\Theta(K')$ and $g_E(\mc{S} \e_K^*, \mc{S} \e_{K'}^*) = 0$, if $\Theta(K) \neq \Theta(K')$. In other words, there exists a subset $\mc{A} \subset \left\{1,...,n\right\}^m$ such that $\{\mc{S}\e_K ~|~ K \in \mc{A}\}$ forms an orthogonal family (not orthonormal though since the elements are not unitary) for the scalar product $g_E$: this subset is chosen of maximal size and so that if $K,K' \in \mc{A}$ with $K \neq K'$, then $\Theta(K)\neq \Theta(K')$. 

\subsubsection{Homogeneous polynomials}

There is a natural identification of $\otimes^m_S E^*$ with the vector space $\mathbf{P}_m(E)$ of \emph{homogeneous polynomials on $E$}, namely polynomials of the form
\[
f(v_1,...,v_m) := \sum_{|\alpha|=m} c_\alpha v_1^{\alpha_1} ... v_n^{\alpha_n}
\]
by considering the isomorphism, for $u \in \otimes_S^m E^*$
\[
\mathbf{P}_m(E) \ni \lambda_m u : v \mapsto u(v,...,v).
\]
Note that $\lambda_m = \lambda_m \mc{S}$, i.e. $\lambda_m$ vanishes on the orthogonal of symmetric tensors (with respect to the metric $g_{\otimes^m E^*}$). In particular, given $K = (k_1,...,k_m) \in \left\{1,...,n\right\}^m$, we have
\[
\lambda_m \left( \mc{S} \e_K^* \right) = \lambda_m \e_K^* = \prod_{j=1}^m v_{k_j}.
\]
The complex/real dimension of $\mathbf{P}_m(E)$ is $p(n,m) := {n+m-1\choose m}$. We denote by $\mathbf{H}_m(E)$ the subspace of harmonic homogeneous polynomials, namely the polynomials $u \in \mathbf{P}_m(E)$ which satisfy the extra condition that $\Delta_E u = 0$, where the Laplacian is computed with respect to the metric $g_E$. On $\mathbf{P}_m(E)$, we introduce the operator
\[
\partial \left(\sum_{|\alpha|=m} c_\alpha v_1^{\alpha_1} ... v_n^{\alpha_n}\right) := \sum_{|\alpha|=m} c_\alpha \partial_{v_1}^{\alpha_1} ... \partial_{v_n}^{\alpha_n},
\]
and we define the scalar product
\[
\langle P, Q \rangle := \partial(P)\overline{Q} \in \C.
\]
Note that if $P = \sum_{|\alpha|=m} a_\alpha v^\alpha, Q = \sum_{|\alpha|=m} b_\alpha v^\alpha$, then 
\[
\langle P,Q\rangle = \sum_{|\alpha|=m} \alpha ! a_\alpha \overline{b}_\alpha.
\]


The operator $\partial$ also satisfies the relation $\partial(PQ) = \partial(P) \partial(Q)=\partial(Q)\partial(P)$. This implies the following:
\begin{lemma}
Let $R \in \mathbf{P}_k(E)$ and let $\Phi : \mathbf{P}_m(E) \rightarrow \mathbf{P}_{m+k}(E)$ be the map defined by $\Phi(Q) = RQ$. Then, $\Phi^* : \mathbf{P}_{m+k}(E) \rightarrow  \mathbf{P}_m(E)$ is given by $\Phi^*(P) = \partial(\overline{R})P$.
\end{lemma}

\begin{proof}
This is straightforward:
\[
\langle \Phi(Q), P \rangle = \langle RQ, P \rangle = \partial(RQ) \overline{P} = \partial(Q) \overline{\partial(\overline{R}) P} = \langle Q, \partial(\overline{R}) P \rangle.
\]
\end{proof}

We have $\partial(|v|^2) = \partial^2_{v_1} + ... + \partial^2_{v_n} = \Delta_E$ and thus:
\[
\langle |v|^2 P, Q \rangle = \langle P, \partial(|v|^2)Q \rangle = \langle P, \Delta_E Q \rangle,
\]
that is the adjoint of $\Delta_E$ with respect to $\langle \cdot, \cdot \rangle$ is $|v|^2$. The map $|v|^2 : \mathbf{P}_m(E) \rightarrow \mathbf{P}_{m+2}(E)$ is clearly injective (and thus $\Delta_E : \mathbf{P}_{m+2}(E) \rightarrow  \mathbf{P}_{m}(E)$ is surjective) and we thus have the decomposition:
\[
\mathbf{P}_m(E) = \mathbf{H}_m(E) \oplus^\bot |v|^2  \mathbf{P}_{m-2}(E) = \oplus_{k \geq 0} |v|^{2k} \mathbf{H}_{m-2k}(E),
\]
where $\mathbf{H}_{m}(E)=\left\{0\right\}$ for $m < 0$. This implies that:
\begin{equation}\label{eq:sphericalharmonicdim}
\mathrm{dim}(\mathbf{H}_m(E)) := h(n,m) = {n+m-1\choose m} - {n+m-3\choose m-2}.
\end{equation}
Moreover, we have:

\begin{lemma}
We have $m (m-1) \lambda_{m-2} \mc{T} = \Delta_E \lambda_m$ and $\lambda_m \mc{J} = |v|^2 \lambda_{m-2}$. As a consequence $\lambda_m : \otimes^m_S E^*|_{0-\Tr} \rightarrow \mathbf{H}_m(E)$ and $\lambda_m : \mc{J}\otimes^{m-2}_S E^*|_{0-\Tr} \rightarrow |v|^2\mathbf{H}_{m-2}(E)$ are both isomorphisms. Moreover, there exists a constant $c_m > 0$ such that if $f \in \otimes^m_S E^*|_{0-\Tr}$, then $\|\lambda_m f\|_{\mathbf{H}_m(E)} = c_m \|f\|_{\otimes^m_S E^*|_{0-\Tr}}$ i.e. $\lambda_m : \otimes^m_S E^*|_{0-\Tr} \rightarrow \mathbf{H}_m(E)$ is (up to a constant factor) an isometry.
\end{lemma}

\begin{proof}
We have for $v \in E$ and $u \in \otimes_S^m E^*$
\[
\lambda_m \mc{J} u(v) = \lambda_m \mc{S} (g_E \otimes u) (v) = \lambda_m (g_E \otimes u) (v) = g_E(v,v)u(v,...,v) = |v|^2\lambda_{m-2} u (v).
\]
Then, we compute for any $m$-tuple $(i_1, \dotso, i_m)$
\begin{align*}
	m(m -1) \lambda_{m-2} \mathcal{T} \mathcal{S} \e_{i_1}^* \otimes ... \otimes \e_{i_m}^* &= \frac{1}{(m-2)!} \sum_{\sigma \in \mathfrak{S}_m} \sum_{j=1}^n \delta_{j i_{\sigma(1)}} \delta_{j i_{\sigma(2)}} v_{i_{\sigma(3)}} \cdots v_{i_{\sigma(m)}}\\
	 &=\sum_{1 \leq k \neq l \leq m} \delta_{i_k i_l} v_{i_1} \cdots \widehat{v}_{i_k} \cdots \widehat{v}_{i_l} \cdots v_{i_m} \\
	 & = \Delta_E \lambda_m \mathcal{S} \e_{i_1}^* \otimes ... \otimes \e_{i_m}^*. 
\end{align*}
This proves the claims made in the first two sentences. The remaining claim follows from Schur's Lemma. Indeed, we have two natural unitary representations of $O(n)$ on $\otimes^m_S E^*$ and $\mathbf{P}_m(E)$ given by the action by pullback (the second action is obvious; for the first one, see the proof of Lemma \ref{lemma:uniform-x-}) and the operator $\lambda_m$ is an intertwining operator. Since $\mathbf{H}_m(E)$ is well-known to be an irreducible representation of $O(n)$, so is $\otimes^m_S E^*|_{0-\Tr}$ and by Schur's Lemma, $\lambda_m^* \lambda_m : \otimes^m_S E^*|_{0-\Tr} \rightarrow \otimes^m_S E^*|_{0-\Tr}$ is a (positive) multiple of the identity.
\end{proof}

\subsubsection{Spherical harmonics}

\label{sssection:spherical-harmonics}

We define the operator of restriction $r_m : \mathbf{P}_m(E) \rightarrow C^\infty(\Ss_E)$, where $\Ss_E$ denotes the unit sphere in $E$ by $r_m(u) := u|_{\Ss_E}$. It is well-known that the operator maps isomorphically $r_m : \mathbf{H}_m(E) \rightarrow \Omega_m$, where $\Omega_m := \ker(\Delta_{\Ss_E} + m(m+n-2))$ and $\Delta_{\Ss_E}$ denotes the Laplacian on the unit sphere of $E$, is an isomorphism. Indeed, this follows from the following formula (see \cite[Proposition 4.48]{Gallot-Hulin-Lafontaine-04} for instance):
\[
\Delta_E(u)|_{\Ss_E} = \Delta_{\Ss_E}(u|_{\Ss_E}) + \left. \dfrac{\partial^2 u}{\partial r^2} \right|_{\Ss_E} + (n-1)\left. \dfrac{\partial u}{\partial r} \right|_{\Ss_E},
\]
where $r$ is the radial coordinate, and using the homogeneity of $u$. We endow $L^2(\Ss_E)$ with the canonical $L^2$ scalar product of functions induced by the round metric, namely
\[
\langle u_1,u_2 \rangle_{L^2(\Ss_E)} := \int_{\Ss_E} u_1(v) \overline{u_2(v)} \dd \vol_{g_{\Ss_E}}(v).
\]
One can prove that up to a constant $c_m' > 0$, $r_m : \mathbf{H}_m(E) \rightarrow \Omega_m$ is an isometry when $\Omega_m$ inherits this metric. We introduce $\pi_m^* := r_m \lambda_m$ and we thus have
\[
\otimes^m_S E^* = \oplus_{k \geq 0} \mc{J}^k \otimes^{m-2k}_S E^*|_{0-\Tr} \rightarrow_{\lambda_m} \mathbf{P}_k(E) = \oplus_{k \geq 0} |v|^{2k} \mathbf{H}_{m-2k}(E) \rightarrow_{r_m} \oplus_{k \geq 0} \Omega_{m-2k}(E)
\]
are isomorphisms which act diagonally on these decompositions. Moreover, they act on each diagonal term (up to a constant factor) as isometries. For the sake of simplicity, we also introduce the notation:
\[
\mathbf{S}_m(E) := \oplus_{k \geq 0} \Omega_{m-2k}(E).
\]

\subsection{Multiplication by a connection $1$-form}

We now twist with a complex inner product space $\mc{E}$ of dimension $r$ and consider the tensor product $\otimes^m_S E^* \otimes \mc{E}$ which consists of elements
\[
f = \sum_{k=1} u_k \otimes e_k,
\]
where $u_k \in \otimes^m_S E^*$ and $(e_1,...,e_r)$ forms an orthonormal basis of $\mc{E}$. Using the map $\lambda_m$ (resp. $\pi_m^*$), we will also identify $\otimes^m_S E^* \otimes \mc{E}$ with $\mathbf{P}_m(E) \otimes \mc{E}$ (resp. $\mathbf{S}_m(E) \otimes \mc{E}$). If $\Gamma \in E^* \otimes \mathrm{End}(\mc{E})$, and $f = \sum_{k=1}^r u_k \otimes e_k \in \mathbf{P}_m(E) \otimes \mc{E}$, we can define $\Gamma f \in  \mathbf{P}_{m+1}(E) \otimes \mc{E}$ by:
\[
\Gamma f (v) = \sum_{k=1} u_k(v) \otimes \Gamma(v)e_k.
\]
The following lemma is standard but we still provide a proof for the sake of completeness:

\begin{lemma}
\label{lemma:multiplication-1-form}
Let $\Gamma \in E^* \otimes \mathrm{End}(\mc{E})$. Then we have the mapping
\[
\mathbf{H}_m(E) \otimes \mc{E} \ni f \mapsto \Gamma f \in (\mathbf{H}_{m+1}(E) \otimes \mc{E}) \oplus (|v|^2 \mathbf{H}_{m-1}(E) \otimes \mc{E}).
\]
\end{lemma}

\begin{proof}
We have
\[
\Gamma f (v) = \sum_{k=1} u_k(v) \otimes \Gamma(v)e_k = \sum_{k,j=1}^r \Gamma_{jk}(v) u_k(v) \otimes e_j,
\]
where $E \ni v \mapsto \Gamma_{jk}(v) \in \C$ is a linear form i.e. a homogeneous polynomial of degree $1$ (which is in particular harmonic). Thus, the lemma boils down to proving that if $u \in \mathbf{H}_m(E)$ and $\eta \in E^*$, then $v \mapsto \eta(v)u(v)$ is an element of $\mathbf{H}_{m+1}(E) \oplus |v|^2 \mathbf{H}_{m-1}(E)$. To see this, define $b:= (n + 2(m -1))^{-1} \nabla \eta \cdot \nabla u \in \mathbf{H}_{m - 1}(E)$, as $\nabla \eta$ is a constant vector so it commutes with $\Delta_E$. Next, we claim that $a := \eta \cdot u - |v|^2 b \in \mathbf{H}_{m + 1}(E)$, so we compute
\[
\begin{split}
\Delta_Ea &= \underbrace{\Delta_E(\eta)}_{=0} u + 2 \nabla \eta \cdot \nabla u + \eta \underbrace{\Delta_E u}_{=0} - \Delta_E(|v|^2) b - 2 \nabla(|v|^2) \cdot \nabla b - |v|^2 \underbrace{\Delta_E b}_{=0}\\
&= 2\nabla \eta \cdot \nabla u - 2nb - 4v \cdot \nabla b = 0,
\end{split}
\]
using Euler's formula since $b$ is $(m-1)$-homogeneous and the definition of $b$. This completes the proof.
\end{proof}

Following the previous lemma, we define $\Gamma_- : \mathbf{H}_m(E) \otimes \mc{E} \rightarrow \mathbf{H}_{m-1}(E) \otimes \mc{E}$ as the orthogonal projection onto the lower-order harmonic polynomials. First of all, we prove the following result, forgetting about the twist by $\mc{E}$ (equivalently $\mc{E} = \C$ in the next lemma).

\begin{lemma}\label{lemma:algebraic-surjectivity-trivial-bundle}
Let $\Gamma \in E^* \setminus \left\{0 \right\}$. Then $\Gamma_- : \mathbf{H}_m(E) \rightarrow \mathbf{H}_{m-1}(E)$ is surjective.
\end{lemma}

\begin{proof}
Up to a preliminary change of coordinates (a rotation), we can always assume that $\Gamma = \mu \e_1^*$ with $\mu \neq 0$. By the previous Lemma, we then have
\[
\Gamma_- u (v) = \mu (n+2(m-1))^{-1} \partial_{v_1}u(v).
\]
Let us compute the dimension of the kernel of $\Gamma_-$. We have $\Gamma_- u =0$ if and only if $\partial_{v_1}u = 0$ i.e. $u$ is independent of $v_1$. Since $u$ is a homogeneous polynomial, this means that $u = \sum_{|\alpha|=m,\alpha_1=0} c_\alpha v^\alpha$. Moreover, since $u$ is harmonic, this also means that $\Delta u = 0 = \Delta' u$, where $\Delta' = \partial^2_{v_2} + ... + \partial^2_{v_n}$ and thus $u$ is harmonic polynomial of degree $m$ in $\mathbf{H}_m(\R^{n-1})$. The other inclusion being obvious, we thus have $\ker \Gamma_- \simeq \mathbf{H}_m(\R^{n-1})$. Thus:
\[
\mathrm{dim}(\ker \Gamma_-) = h(n-1,m) = {n-2+m\choose m} - {n-4+m\choose m-2}.
\]
As a consequence, using the Pascal's rule for binomial coefficients
\[
\begin{split}
\mathrm{dim}(\ran \Gamma_-) & = \mathrm{dim}(\mathbf{H}_m(\R^{n})) - \mathrm{dim}(\ker \Gamma_-) \\
& = {n-1+m \choose m} - {n-3+m \choose m-2} - \left( {n-2+m\choose m} - {n-4+m\choose m-2} \right) \\
& = {n-2+m \choose m-1} - {n-4+m \choose m-3} = \mathrm{dim}(\mathbf{H}_{m-1}(\R^n)),
\end{split}
\]
thus $\Gamma_-$ is surjective.
\end{proof}

Note that, using the restriction map $r_m : \mathbf{H}_m(E) \rightarrow \Omega_m(E)$, the last lemma is equivalent to saying that $\Gamma_- : \Omega_m \rightarrow \Omega_{m-1}$ is surjective. Eventually, we will need this last Lemma, where $\mathrm{End}_{\mathrm{sk}}(\mc{E})$ denotes the skew-Hermitian endomorphisms

\begin{lemma}
\label{lemma:algebraic-surjectivity}
Let $u \in \Omega_m \otimes \mc{E}$. Then, there exists $\Gamma \in \mathrm{End}_{\mathrm{sk}}(\mc{E})$ and $w \in \Omega_{m+1} \otimes \mc{E}$ such that $u = \Gamma_- w$.
\end{lemma}

\begin{proof}
We write $u = \sum_{k=1}^r u_k \otimes e_k$, where $u_k \in \Omega_m(E)$ are spherical harmonics (possibly complex). For each $k = 1,...,r$, we choose an arbitrary real-valued $\Gamma_k \in E^* \setminus \left\{ 0 \right\}$ and we define $\Gamma \in E^* \otimes \mathrm{End}_{\mathrm{sk}}(\mc{E})$ by, in the $(e_1, \dotso, e_r)$ basis:
\[
\Gamma(v) := \begin{pmatrix} i \Gamma_1(v) & 0 & \cdots & 0 \\ 
0 & i \Gamma_2(v) & 0 & \cdots \\
\vdots & & \ddots & \\
0& \cdots & 0 & i \Gamma_r(v)
\end{pmatrix},
\]
where $v \in E$. By the previous Lemma, for all $k = 1,...,r$ we can always find $w_k \in \Omega_{m+1}$ such that $i{\Gamma_k}_- w_k = u_k$. We then set $w := \sum_{k = 1}^r w_k \otimes e_k$, so $\Gamma_- w = u$.

\end{proof}

\subsection{Twisted tensor analysis on the manifold}

\label{ssection:twisted-tensors}

Given a section $u \in C^\infty(M,\otimes^m_S T^*M \otimes \mc{E})$, the connection $\nabla^{\mc{E}}$ produces an element $\nabla^{\mc{E}}u \in C^\infty(M, T^*M \otimes (\otimes^m_S T^*M) \otimes \mc{E})$. In coordinates, if $(e_1, ..., e_r)$ is a local orthonormal frame for $\mc{E}$ and $\nabla^{\mc{E}} = d + \Gamma$, for some one-form with values in skew-hermitian matrices $\Gamma$, we have:
\begin{equation}
\label{equation:nabla-e}
\begin{split}
\nabla^{\mc{E}}(\sum_{k=1}^r u_k \otimes e_k) & = \sum_{k=1}^r \nabla u_k \otimes e_k + u_k \otimes \nabla^{\mc{E}} e_k\\
& = \sum_{k=1}^r \left(\nabla u_k +   \sum_{l=1}^r \sum_{i=1}^n \Gamma_{il}^k u_l \otimes dx_i \right) \otimes e_k,
\end{split}
\end{equation}
where $u_k \in C^\infty(M,\otimes^m_S T^*M)$ and $\nabla$ is the Levi-Civita connection. The symmetrization operator $\mc{S}_{\mc{E}} : C^\infty(M,\otimes^m T^*M \otimes \mc{E}) \rightarrow C^\infty(M,\otimes^m_S T^*M \otimes \mc{E})$ is defined by:
\[
\mc{S}_{\mc{E}}\left(\sum_{k=1}^r u_k \otimes e_k\right) = \sum_{k=1}^r \mc{S}(u_k) \otimes e_k,
\]
where $u_k \in C^\infty(M,\otimes^m_S T^*M)$ and $\mc{S}$ is the symmetrization operators of tensors previously introduced. We can symmetrize \eqref{equation:nabla-e} to produce an element $D_{\mc{E}} := \mc{S}_{\mc{E}} \nabla^{\mc{E}}u \in C^\infty(M, \otimes^{m+1}_S T^*M \otimes \mc{E})$ given in coordinates by:
\begin{equation}
\label{equation:formula-de}
D_{\mc{E}} \left(\sum_{k=1}^r u_k \otimes e_k\right) = \sum_{k=1}^r \left( Du_k + \sum_{l=1}^r  \sum_{i=1}^n \Gamma_{il}^k \mc{S}(u_l \otimes dx_i) \right) \otimes e_k,
\end{equation}
where $D := \mc{S} \nabla$ ($\nabla$ being the Levi-Civita connection) is the usual symmetric derivative of symmetric tensors. Elements of the form $Du \in C^\infty(M,\otimes^{m+1}_S T^*M)$ are called \emph{potential tensors}. By comparison, we will call elements of the form $D_{\mc{E}}f \in C^\infty(M,\otimes^{m+1}_S T^*M \otimes \mc{E})$ \emph{twisted potential tensors}. The operator $D_{\mc{E}}$ is a first order differential operator and the expression of its principal symbol
\[
\sigma_{\mathrm{princ}}(D_\mc{E}) \in C^\infty(T^*M, \mathrm{Hom}(\otimes^{m}_S T^*M \otimes \mc{E}, \otimes^{m+1}_S T^*M \otimes \mc{E}))
\]
can be read off from \eqref{equation:formula-de}, namely $\sigma_{\mathrm{princ}}(D_{\E}) = \sigma_{\mathrm{princ}}(D) \otimes \Id_{\E}$:
\[
\begin{split}
\sigma_{\mathrm{princ}}(D_\mc{E})(x,\xi) \cdot \left(\sum_{k=1}^r u_k(x) \otimes e_k(x) \right) & = \sum_{k=1}^r \left(\sigma_{\mathrm{princ}}(D)(x,\xi) \cdot u_k(x)\right) \otimes e_k(x) \\
& = i \sum_{k=1}^r \mc{S}(\xi \otimes u_k(x)) \otimes e_k(x),
\end{split}
\]
where $e_k(x) \in \mc{E}_x, u_k(x) \in \otimes^m_S T^*_xM$ and the basis $(e_1(x),...,e_r(x))$ is assumed to be orthonormal. One can check that this is an injective map, which means that $D_{\mc{E}}$ acting on twisted symmetric tensors of order $m$ is a left-elliptic operator and can be inverted on the left modulo a smoothing remainder; its kernel is finite-dimensional and consists of elements called \emph{twisted Killing Tensors}.

In the particular case where $\mc{E} = \C$ (i.e. there is no twist), the elements in the kernel of $D$ are called \emph{Killing Tensors} (for $m=1$, they generate infinitesimal isometries). It is known that if the flow is ergodic, the kernel of $D$ is trivial in the sense that it is reduced to $\left\{0\right\}$ when $m$ is odd and $\C \cdot \mc{S}(g^{\otimes m/2})$ when $m$ is even. This simply follows from the well-known conjugation relation $\pi_{m+1}^*D = X \pi_m^*$. Moreover, it is known that the kernel of $D$ is generically trivial \cite{Kruglikov-Matveev-16} (with respect to the metric $g$). In the presence of a twist by a vector bundle $\mc{E}$, one can also analyse the kernel of $D_{\mc{E}}$: it is proved in \cite{Guillarmou-Paternain-Salo-Uhlmann-16} that on a negatively-curved manifold, if $(\mc{E},\nabla^{\mc{E}})$ has no CKTs, then the kernel of $D_{\mc{E}}$ is trivial (in the same sense as before) 
This also relies on the conjugation relation $\pi_{m+1}^* D_{\mc{E}} = \X \pi_m^*$.

The adjoint
\[
D^*_{\mc{E}} : C^\infty(M, \otimes^{m+1}_S T^*M \otimes \mc{E}) \rightarrow C^\infty(M, \otimes^{m}_S T^*M \otimes \mc{E})
\]
has a surjective principal symbol given by
\begin{equation}\label{eq:D*symbol}
\sigma_{D^*_{\mc{E}}}(x,\xi) \cdot \left(\sum_{k=1}^r u_k(x) \otimes e_k(x) \right) =- i \sum_{k=1}^r \imath_{\xi^\sharp} u_k(x) \otimes e_k(x).
\end{equation}
As we will see in the next section (see Definition \ref{definition:divergence-type}), such an operator is called of divergence type. Using the correspondence between trace-free twisted symmetric tensors of degree $m$ and twisted spherical harmonics of degree $m$, there is an explicit link between $\X_-/D^*_{\mc{E}}$ and $\X_+/D_{\mc{E}}$. More precisely, we introduce $\mc{P} : C^\infty(M,\otimes^m_S T^*M \otimes \mc{E}) \rightarrow C^\infty(M,\otimes^m_S T^*M|_{0-\Tr} \otimes \mc{E})$ the pointwise orthogonal projection on trace-free twisted symmetric tensors. We then have the following equalities (see \cite[p. 22]{Guillarmou-Paternain-Salo-Uhlmann-16}) on $C^\infty(M, \otimes^{m}_S T^*M|_{0-\Tr} \otimes \mc{E})$:
\begin{equation}
\label{equation:link-xd}
 \X_+ \pi_m^* = \pi_{m+1}^* \mc{P} D_{\mc{E}}, \quad \X_- \pi_{m}^* = - \dfrac{m}{n-2+2m} \pi_{m-1}^* D^*_{\mc{E}}.
\end{equation}
The kernel of $\X_+$ is therefore in one-to-one correspondance with the kernel of $\mc{P}  D_{\mc{E}}$. In particular, we have the mapping
\[
D^*_{\mc{E}} : C^\infty(M, \otimes^{m+1}_S T^*M|_{0-\Tr} \otimes \mc{E}) \rightarrow C^\infty(M, \otimes^{m}_S T^*M|_{0-\Tr} \otimes \mc{E}).
\]

\subsection{Trace free skew-hermitian endomorphisms}

The following lemma is a very elementary result of linear algebra: it states that every trace-free skew-Hermitian endomorphism can be obtained as the Lie bracket of two skew-Hermitian endomorphisms.

\begin{lemma}
\label{lemma:baby}
Let $u \in \mathrm{End}_{\mathrm{sk}}(\C^r)$ such that $\Tr(u) = 0$. Then, there exists $A,\Gamma \in \mathrm{End}_{\mathrm{sk}}(\C^r)$ such that $u = [A,\Gamma]$.
\end{lemma}

Note that the same holds true if one removes the skew-Hermitian subscript, namely any trace-free matrix is a commutator.

\begin{proof}
The proof is based on an induction on the dimension $r \geq 0$. The statement is obviously true for $r=0,1$. Fix $r \geq 1$ and assume it is true for $r-1$. Consider the map
\[
F : \Ss^{r-1} \ni x \mapsto \langle u(x),x \rangle \in i\R,
\]
where $\Ss^{r-1} := \left\{x \in \C^r ~|~ \langle x, x \rangle = 1 \right\}$. Note that it maps to $i \R$ since $u$ is skew-hermitian. Since $\Ss^{r-1}$ is connected and $F$ is continuous, the image $F(\Ss^{r-1}) \subset i\R$ is an interval: its maximum/minimum is the highest/lowest eigenvalue of $u$ (modulo multiplication by $i$). Since $\Tr(u) = 0$, this interval contains $0$ so there exists $x_0 \in \Ss^{r-1}$ such that $\langle u(x_0),x_0 \rangle = 0$. We can therefore find an orthonormal basis of $\C^r$ so that in this basis,
\[
u = \begin{pmatrix}
0 & -X^* \\
X & u'
\end{pmatrix},
\]
where $X$ is an $(r-1)$-dimensional column and $X^* := -\overline{X}^{\top}$ is the conjugate transpose. Note that $u' \in \mathrm{End}_{\mathrm{sk}}(\C^{r-1})$, $\Tr (u') = 0$ and ${u'}^* = -u'$, thus $u' = [A',\Gamma']$ for some $A',\Gamma' \in \mathrm{End}_{\mathrm{sk}}(\C^{r-1})$. Let $\lambda \in \mathbb{R}$ be such that $A'-i \lambda$ is invertible and consider $S \in \C^{r-1}$ such that $(A'-i \lambda)S = X$. Consider
\[
A = \begin{pmatrix} i \lambda & 0 \\ 0 & A' \end{pmatrix},\,\, \Gamma = \begin{pmatrix} 0 & - S^* \\ S & \Gamma' \end{pmatrix},
\]
and observe that
\[
[A,\Gamma] = A\Gamma-\Gamma A = \begin{pmatrix} 0 &S^*(A'-i \lambda) \\ (A'-i\lambda)S & [A',\Gamma'] \end{pmatrix} = \begin{pmatrix} 0 & -X^* \\ X & u' \end{pmatrix} = u.
\]
This completes the proof.
\end{proof}

\section{Microlocal preliminaries}

\label{section:microlocal}

\subsection{An abstract result} In this section, we introduce the notion of \emph{operators of uniform divergence type.}

\subsubsection{Statement of the result}

Let $P : C^\infty(M,E) \rightarrow C^\infty(M,F)$ be a differential operator of order $m \geq 0$ between two vector bundles such that $\rk(F) > \rk(E)$ and let $\sigma_P \in C^\infty(T^*M, \mathrm{Hom}(E,F))$ be its principal symbol.

\begin{definition}
\label{definition:divergence-type}
We say that $P$ is of \emph{gradient type} (or equivalently that $P^*$ is of \emph{divergence type}) if $\sigma_P(x,\xi)$ is injective for all $(x,\xi) \in T^*M \setminus \left\{ 0 \right\}$ (equivalently, $\sigma_{P^*}(x,\xi)$ is surjective for all $(x,\xi) \in T^*M \setminus \left\{ 0 \right\}$).
\end{definition}

\begin{lemma}\label{lemma:infkernel}
Assume that $P$ is of gradient type. Then for all $s \in \R$, $\ker(P^*|_{H^s(M,F)})$ is infinite dimensional.
\end{lemma}

\begin{proof}
Injectivity of the principal symbol implies the existence of a pseudodifferential operator $Q \in \Psi^{-m}(M;E,F)$ such that $QP = \mathbbm{1}_{E} + R$, where $R$ is a smoothing operator. By classical arguments, this implies that for any $s \in \R$, the image $P(H^{s+m}(M,E)) \subset H^s(M,F)$ is closed. This implies the decomposition
\[
H^s(M,F) = \ker(P^*|_{H^s(M,F)}) \oplus^\bot P(H^{s+m}(M,E)),
\]
which is orthogonal for the $L^2$ scalar product. By ellipticity, the kernel of $P$ is finite dimensional. We introduce the formally self-adjoint operator $\Delta := P^*P$ and denote by $\Pi_0$ the $L^2$-orthogonal projection on $\ker \Delta = \ker P$. Thus, any section $f \in H^s(M,F)$ can be uniquely decomposed as $f = Pu + v$, where $v \in \ker(P^*|_{H^s(M,F)})$ and $u \in H^{s+m}(M,E) \cap \ker \Pi_0$. The $L^2$-orthogonal projection on the image of $P$ is a self-adjoint pseudodifferential operator of order $0$, defined by
\[
\pi_{\ran(P)} := P \Delta^{-1} P^*,
\]
where $\Delta^{-1}$ is the operator defined by the $0$ in restriction to $\ran(\Pi_0)$ and by the inverse of $\Delta$ on $\ker(\Pi_0)$. The $L^2$-orthogonal projection on the kernel of $P^*$ is then given by $\pi_{\ker P^*} := \mathbbm{1}_F - \pi_{\ran(P)}$. Note that $f \in \ker (P^*|_{H^s(M,F)})$ if and only if $\pi_{\ran(P)}f = 0$.

We first show that $\ker(P^*|_{H^s(M,F)}) \neq \left\{ 0 \right\}$. Assume it is not the case, that is any $f \in H^s(M,F)$ is of the form $f = Pu$, where $u \in H^{s+m}(M,E)$. We can then consider for $h > 0$ and $(x_0,\xi_0) \in T^*M$ a section $f \in C^\infty(M,F)$ such that $0 \neq f(x_0) \in \ker \sigma_{\pi_{\ran(P)}}(x_0,\xi_0)$ (note that $\sigma_{\pi_{\ran(P)}}(x_0,\xi_0)$ is a symbol of order $0$; it is the orthogonal projection on the image $\sigma_P(x_0,\xi_0)(E_{x_0}) \subset F_{x_0}$). This is always possible since $\rk(F) > \rk(E)$. We further assume that $\|f(x)\|_{F} = 1$ for all $x$ in a neighborhood of $x_0$. We consider a Lagrangian state $e^{\frac{i}{h} S}$ such that $S(x_0)=0, \dd S(x_0) = \xi_0$. Then, we have $\pi_{\ran(P)}(e^{\frac{i}{h}S}f)(x_0) = \sigma_{\pi_{\ran(P)}}(x_0,\xi_0)f(x_0) + \mc{O}(h) = \mc{O}(h).$ But we have:
\[
1 = \|e^{\frac{i}{h} S}f(x_0)\|_{F_{x_0}} = \|(\pi_{\ran(P)}(e^{\frac{i}{h} S}f))(x_0)\|_{F_{x_0}} = \mc{O}(h),
\]
which is a contradiction.


We now assume that $\ker(P^*|_{H^s(M,F)})$ is finite dimensional. Writing $\mathbbm{1}_F = \pi_{\ran(P)} + \pi_{\ker(P^*)}$, we can construct a Gaussian state\footnote{Here, in local coordinates, $e_{x_0,\xi_0}$ has the form:
\[
e_{x_0,\xi_0}(x) = (\pi h)^{-n/4} e^{\frac{i}{h}\xi_0 \cdot (x-x_0) - \frac{1}{2h}|x-x_0|^2}.
\]} $\varphi e_{x_0,\xi_0} f$, where $\varphi$ is a local cut-off function with $\varphi = 1$ near $x_0$. We assume $f(x_0) \in \ker \sigma_{\pi_{\ran(P)}}(x_0,\xi_0)$ and $\|f\|_F = 1$ close to $x_0$. It can be checked that $\|\varphi e_{x_0, \xi_0}(h) f\|_{L^2} = c + o(1)$, for some $c > 0$. Moreover, a computation in local coordinates gives (see also \cite[Equation 1.5]{Duistermaat-Guillemin-75} and \cite[p. 102]{Zworski-12})
\begin{multline*}
\|\pi_{\ran (P)} \varphi e_{x_0, \xi_0} (h) f\|_{L^2}^2 = \langle{\pi_{\ran (P)} \varphi e_{x_0, \xi_0} (h) f, \varphi e_{x_0, \xi_0} (h) f}\rangle_{L^2}\\
= \langle{\sigma_{\pi_{\ran (P)}}(x_0, \xi_0) f(x_0), f(x_0)}\rangle_{F_{x_0}} + o(1) = o(1),
\end{multline*}
and so we obtain
\[
\varphi e_{x_0,\xi_0}(h) f = \underbrace{\pi_{\ran(P)} \varphi e_{x_0,\xi_0}(h)f}_{o_{L^2}(1)} + \pi_{\ker(P^*)} \varphi e_{x_0,\xi_0}(h)f.
\]
Since $\ker(P^*)$ is finite dimensional, we can always assume that $\pi_{\ker(P^*)} \varphi e_{x_0,\xi_0}(h)f \rightarrow_{h \rightarrow 0} v \in \ker(P^*|_{H^s(M,F)})$, that is $\varphi e_{x_0,\xi_0}(h)f \rightarrow v\in \ker(P^*|_{H^s(M,F)})$ (the convergence takes place in $L^2$ but the limit $v$ is in $H^s$). But this can always be achieved by taking an arbitrary large number of such $\varphi_i e_{x_i,\xi_i}(h)f_i$, $i=1,...,N$ with disjoint supports on the manifold ($\varphi_i$ is supported near $x_i$). This produces non-zero elements $v_i \in  \ker(P^*|_{H^s(M,F)})$ which are all pairwise orthogonal, contradicting the finite-dimensionality of $\ker(P^*)$.
\end{proof}

We introduce the following property.

\begin{definition}
\label{definition:uniform-divergence-type}
We say that $P^*$ is \emph{uniformly of divergence type} if it is of divergence type and for all $x_0 \in M$:
\[
\Sigma_{|\xi|=1} \ker \sigma_{P^*}(x_0,\xi) = F_{x_0}.
\]
\end{definition}

Note that $\ker \sigma_{P^*}(x_0,\xi) = \ker \sigma_{\pi_{\ran(P)}}(x_0,\xi)$. The restriction $\xi \neq 0$ is due to the fact that the principal symbol is $0$-homogeneous in $\xi$ and thus only makes sense for large $\xi$. We then have the following result:

\begin{lemma}
\label{lemma:microlocal-surjectivity}
Assume $P$ is of gradient type. Let $s > n/2$, $x \in M$ and define the map $\mathrm{ev}_x : \ker P^*|_{H^s(M,F)} \rightarrow F_{x}$ by $\mathrm{ev}_x(f) := f(x)$. Then:
\[
\mathrm{ev}_x : \ker P^*|_{H^s(M,F)} \rightarrow \Sigma_{|\xi|=1} \ker \sigma_{P^*}(x_0,\xi)
\]
is surjective. In particular, if $P^*$ is of uniform divergence type, then $\mathrm{ev}_x : \ker P^*|_{H^s(M,F)} \rightarrow F_{x}$ is surjective.
\end{lemma}

The choice of $s > n/2$ is simply there to ensure that $H^s$ embeds continuously into $C^0$ and thus $\mathrm{ev}_x$ is well-defined. The previous Lemma gives a lower bound on the possible values that elements in $\ker P^*$ can take at a given point.

\begin{proof}
Fix $x_0 \in M$. It is sufficient to prove that for $\xi \in T^*M \setminus \left\{ 0 \right\}$, one has $\ker \sigma_{P^*}(x_0,\xi) \subset \ran(\mathrm{ev})$. We consider a Lagrangian state $f(h) := e^{\frac{i}{h} S} f$, for some smooth section $f$ (independent of $h > 0$) where $\dd S(x_0) = \xi,S(x_0)=0$, and $f(x_0) \in \ker \sigma_{P^*}(x_0,\xi)$. Then
\[
\begin{split}
f(x_0) = (f(h))(x_0) & = (\pi_{\ran(P)}f(h))(x_0) + (\pi_{\ker(P^*)}f(h))(x_0) \\
& = \underbrace{\sigma_{\pi_{\ran(P)}}(x_0,\xi) f(x_0)}_{=0} + \mc{O}(h) + \underbrace{(\pi_{\ker(P^*)}f(h))(x_0)}_{\in \ran(\mathrm{ev}_{x_0})}.
\end{split}
\]
Composing with the orthogonal projection $\pi_{\ran(\mathrm{ev}_{x_0})^\bot}$ on $\ran(\mathrm{ev}_{x_0})^\bot$, we then obtain that pointwise at $x_0$:
\[
\pi_{\ran(\mathrm{ev}_{x_0})^\bot} f(x_0) = \mc{O}(h),
\]
and thus, since $f(x_0)$ is independent of $h$, $\pi_{\ran(\mathrm{ev}_{x_0})^\bot} f(x_0) = 0$ i.e. $f(x_0) \in \ran(\mathrm{ev}_{x_0})$.
%
%
%
%
\end{proof}

\subsubsection{Example: the divergence of a vector field}

Let us illustrate the preceding property by a simple example i.e. the divergence of a vector field. Let $(M,g)$ be a smooth Riemannian manifold. Given $X \in C^\infty(M,TM)$, the divergence $\delta X \in C^\infty(M)$ of $X$ is defined as minus the $L^2$ formal adjoint of the gradient operator, namely:
\[
\langle f, \delta X \rangle := - \langle \nabla f, X \rangle.
\]

\begin{lemma}\label{lemma:divergence}
The divergence operator $\delta$ is of uniform divergence type.
\end{lemma}


\begin{proof}
We first prove that $\delta$ is of divergence type. For that, it is sufficient to compute its principal symbol. It is an elementary computation to show that $\sigma_\nabla(x,\xi) = i \times \xi^\sharp$ and thus for $v \in T_xM$, one has $\sigma_{\delta}(x,\xi)v = i \langle \xi,v \rangle$. Observe that $\sigma_\nabla(x,\xi)$ is injective for all $\xi \neq 0$ with constant rank equal to $1$, i.e. $\delta$ is of divergence type. We now show that $\delta$ satisfies Definition \ref{definition:uniform-divergence-type}. Pick $x_0 \in M, v \in T_xM$ and consider $\xi \in T_{x_0}^*M \setminus \left\{ 0 \right\}$ such that $\xi^\sharp \bot v$. Then:
\[
\sigma_\delta(x_0,\xi)v = i \langle \xi, v \rangle =  i g(\xi^\sharp, v)  = 0.
\]
\end{proof}

\subsubsection{Example: differential forms} More generally, consider the bundle of differential $k$-forms, $\Omega^k = \Lambda^k T^*M$, the exterior derivative $d$ and its formal adjoint $d^*$ acting on sections of $\Omega^k$. It can be checked that for $\alpha \in \Omega^k(x)$
\[\sigma_d(x, \xi)\alpha = i \xi \wedge \alpha, \quad \sigma_{d^*}(x, \xi) \alpha = -i \iota_{\xi^\sharp} \alpha.\]
In fact one may show $\ker \sigma_d(x, \xi)|_{\Omega^k(x)} = \xi \wedge \Omega^{k-1}(x)$, so $d$ is of gradient type if and only if $k = 0$. Equivalently $d^*$ is of divergence type if and only if $k = 1$, which by metric duality is the content of Lemma \ref{lemma:divergence}. Again by duality, we obtain $\ker \sigma_{d^*}(x, \xi)|_{\Omega^k(x)} = \iota_{\xi^\sharp} \Omega^{k+1}(x)$ and since pointwise in local coordinates every differential $k$-form $dx_{i_1} \wedge \dotso \wedge dx_{i_k}$ is obtained by contracting a suitable $(k+1)$-form, we obtain
\begin{equation}\label{eq:diffforms}
	\Sigma_{|\xi| = 1} \ker_{d^*}(x, \xi) = \Omega^k(x).
\end{equation}
Observe that Lemma \ref{lemma:infkernel} does not apply directly to $d$ but by the Hodge decomposition $\ker d^*|_{\Omega^k} = d^*C^\infty(M; \Omega^{k+1}) \oplus \mathcal{H}^k$ is infinite dimensional, where $\mathcal{H}^k$ are harmonic $k$-forms. However, setting $\Delta = dd^* + d^*d$ in the proof of the same lemma would produce the analogous result with minor corrections. Finally, Lemma \ref{lemma:microlocal-surjectivity} also does not apply directly, but by using the Hodge decomposition and \eqref{eq:diffforms} we obtain the analogous result: for $x \in M$, the map $\mathrm{ev}_x: \ker d^*|_{H^s(M; \Omega^k)} \to \Omega^k(x)$ is surjective for $s > n/2$.

\subsubsection{Counterexample: a divergence type operator that is not uniform}

This is a very elementary example constructed by hand so that it does not work, but it is very likely that one can find more elaborate examples. Consider for $(M,g)$ a smooth Riemannian manifold and a vector bundle $\mc{E} \rightarrow M$ over $M$. Consider an elliptic selfadjoint differential operator $P : C^\infty(M,\mc{E}) \rightarrow C^\infty(M,\mc{E})$ and assume $P$ is invertible. Let $Q : C^\infty(M,\mc{E}) \rightarrow C^\infty(M,\mc{E} \oplus \mc{E})$ defined by $Qf := (Pf, -Pf)$, then $\sigma_Q(x,\xi)u = (\sigma_P(x,\xi)u, -\sigma_P(x,\xi)u)$ and $\sigma_{Q^*}(x,\xi)(u_1,u_2) = \sigma_P(x,\xi)(u_1-u_2)$. Thus $Q$ is of gradient type or equivalently $Q^*$ is of divergence type. But $Q^*$ is not of uniform divergence type since 
\[
\Sigma_{|\xi|=1} \ker \sigma_{Q^*}(x,\xi) = \left\{ (u, u) ~|~ u \in \mc{E}_x \right\} \simeq \mc{E}_x \neq \mc{E}_x \oplus \mc{E}_x.
\]
In particular, it is easy to describe the kernel of $Q^*$ since $P$ is invertible, namely
\[
\ker Q^*|_{C^\infty(M,\mc{E}\oplus \mc{E})} = \left\{(f,f) ~|~ f \in C^\infty(M,\mc{E})\right\},
\]
and thus for every $x \in M$, the map $\mathrm{ev}_x : \ker Q^*|_{C^\infty(M,\mc{E}\oplus \mc{E})} \rightarrow \mc{E}_x \oplus \mc{E}_x$ defined by $\mathrm{ev}_x(f_1,f_2) = (f_1(x),f_2(x))$ is not surjective. Note that this example shows that the lower bound given by Lemma \ref{lemma:microlocal-surjectivity} is sharp. Also observe that, taking $Qf = (\Delta f, -\Delta f) \in C^\infty(M,\C^2)$, where $\Delta : C^\infty(M,\C) \rightarrow C^\infty(M,\C)$ is the Laplacian induced by $g$ and acting on functions, one obtains an operator which is divergence type but not uniform. However this time, the map $\mathrm{ev}_x$ is surjective for every $x \in M$. This comes from the fact that the kernel of $\Delta$ is not trivial (and given by the constants).

\subsection{Application to trace-free divergence-free tensors}

We now study the operators $\X_+$ and $\X_-$ (see \eqref{eq:XX}) in the light of the preceding paragraph. We first have the

\begin{lemma}
The operator $\X_-$ is of divergence type.
\end{lemma}

\begin{proof}
By definition, it is sufficient to prove that $\X_+$ is of gradient type i.e. that its principal symbol is injective. By \eqref{equation:link-xd}, the principal symbol of $\X_+$ is given (up to conjugating by the map $\pi_m^*$) by
\[
\sigma_{\X_+}(x,\xi) \left( \sum_{k=1}^r u_k(x) \otimes e_k(x) \right) = \sum_{k=1}^r  i \mc{P} \mc{S}(\xi \otimes u_k(x)) \otimes e_k(x),
\]
where $u_k(x) \in \otimes^m_S T^*_xM|_{0-\Tr}$. Thus, it is sufficient to prove that
\[
\otimes^m_S T^*M|_{0-\Tr} \ni u \mapsto \mc{P} \mc{S}(\xi \otimes u)
\]
is injective. This is the content of \cite[Theorem 5.1]{Dairbekov-Sharafutdinov-10}; it can also be found in \cite{Guillemin-Kazhdan-80-2}.
\end{proof}

As a direct application of the preceding paragraph, we obtain that $\ker \X_-|_{H^s(M,\Omega_m)}$ is infinite-dimensional for all $s \in \R$. We also have:

\begin{lemma}
\label{lemma:uniform-x-}
The operator $\X_-$ is of uniform divergence type if and only if $n \geq 3$.
\end{lemma}

As a consequence of the previous paragraph, for all $s > n/2$ and $x_0 \in M$, the map $\mathrm{ev}_{x_0} : \ker \X_-|_{H^s(M,\Omega_m \otimes \mc{E})} \rightarrow \Omega_m  \otimes \mc{E}(x_0)$ defined by $\mathrm{ev}(w) := w(x_0)$ is surjective.

%
%
\begin{proof}
We use that according to \eqref{equation:link-xd}, the operator $\X_-$ on $C^\infty(M,\Omega_m \otimes \mc{E})$ is equivalent to the operator $D^*_{\mc{E}}$ acting on $C^\infty(M,\otimes^m_S T^*M|_{0-\Tr} \otimes \mc{E})$. Since the principal symbol of the operator acts diagonally on $\mc{E}$, it is sufficient to prove it for $\mc{E}=\C$, i.e. there is no twist.

It is classical that $D^* \sim \eta_-$ is elliptic if $n = 2$ (see e.g. \cite{Guillemin-Kazhdan-80}) so its symbol is injective and so $D^*$ is not of uniform divergence type. Now assume $n \geq 3$ and recall that $\sigma(D^*)(x, \xi) = - i \iota_{\xi^\sharp}: \Omega_m(x) \to \Omega_{m -1}(x)$. Note that $\dim \ker \iota_{\xi^\sharp} > 0$ by dimension counting \eqref{eq:sphericalharmonicdim}. Consider the subspace
	\[W := \Sigma_{|\xi|=1} \ker \iota_{\xi^\sharp}|_{\Omega_m(x)} \subset \Omega_m(x).\]
	We claim first that $W$ is invariant under the action of $O(n)$. To see this, let $A \in O(n)$; it suffices to show that $A \ker \iota_{\xi^\sharp} \subset \ker \iota_{A\xi^\sharp}$. Let $s \in \ker \iota_{\xi^\sharp}|_{\Omega_m(x)}$ and denote by $\e_1, \dotso, \e_n$ an orthonormal basis at $T_xM$
	\[s = \sum_I s_I \e_{i_1}^* \otimes \dotso \otimes \e_{i_m}^*, \quad I = (i_1, \dotso, i_m) \subset \left\{1,...,n\right\}^m.\]
	Note simply that $A\xi^\sharp = (A \xi)^\sharp$, where $A\xi = \xi \circ A^T$ is the left group action so
	\begin{align}
	\begin{split}\label{eq:invarianceOn}
		\iota_{A\xi^\sharp} As &= \sum_I s_I \iota_{A\xi^\sharp} (A\e_{i_1})^* \otimes \dotso \otimes (A\e_{i_m})^*\\
		&= \sum_I s_I \langle{A \xi^\sharp, A\e_{i_1}}\rangle_x (A\e_{i_2})^* \otimes \dotso \otimes (A\e_{i_m})^*\\
		&= A\sum_I s_I \langle{{\xi^\sharp}, \e_{i_1}}\rangle_x(\e_{i_2})^* \otimes \dotso \otimes (\e_{i_m})^* = A\iota_{\xi^\sharp}s.
	\end{split}
	\end{align}
	Here, we used that by definition $A$ preserves the inner product. This proves the observation and so $W \neq \{0\}$ is a sub-representation of $\Omega_m$. But it is well-known that the representation of $O(n)$ on $\Omega_m(x)$ is irreducible, thus $W = \Omega_m(x)$ completing the proof.
\end{proof}

It is straightforward to extend this claim to all symmetric tensors of some order.
\begin{proposition}\label{prop:D_Euniformdivergence}
The operator $D^*_{\E}$ acting on \emph{all} symmetric tensors $C^\infty(M; \otimes_S^m T^*M \otimes \E)$ is of uniform divergence type.
\end{proposition}
\begin{proof}
It suffices to consider $\E = \mathbb{C}$. Next, it is sufficient to recall the decomposition in \eqref{eq:symetrictensorsplitting}: for any $k$, consider 
\[\{0\} \neq W := \Sigma_{|\xi| = 1} \ker \iota_{\xi^\sharp}|_{\mathcal{J}^k \otimes_S^{m - 2k} T_x^*M|_{0-\Tr}} \subset \mathcal{J}^k \otimes_S^{m - 2k} T_x^*M |_{0-\Tr} \equiv \Omega_{m - 2k}(x).\]
One checks that $O(n)$ acts on the left on $\mathcal{J}^k \otimes_S^{m - 2k}T_x^*M|_{0-\Tr} $ (as $g(A\cdot, A\cdot) = g(\cdot, \cdot)$ for $A \in O(n)$) via its action on $\Omega_{m - 2k}(x)$ and the computation in \eqref{eq:invarianceOn} remains valid to show $O(n)$ acts on $W$. As the representation of $O(n)$ on $\Omega_{m - 2k}(x)$ is irreducible we get $W = \Omega_{m - 2k}(x)$, proving the claim.
\end{proof}

\section{Generic absence of CKTs}

\subsection{Perturbation of the Laplacian}

\label{ssection:perturbation}

Consider a connection $\nabla^{\mc{E}}$ with CKTs. We denote by $\X^\Gamma := \X + \Gamma(v)$ the operators induced by the unitary connections $\nabla^{\mc{E}} + \Gamma$, where $\Gamma \in C^\infty(M,T^*M \otimes \mathrm{End}_{\mathrm{sk}}(\mc{E}))$ is small enough. We introduce $\Delta^\Gamma_+ := -  \X^\Gamma_- \X^\Gamma_+ \geq 0$. Each $\Delta^\Gamma_+$ is a Laplacian type operator in the sense that it is non-negative, formally selfadjoint (and with principal symbol given by $\sigma_{\Delta^{\Gamma}_+}(x,\xi) = |\xi|^2\mathbbm{1}_{\Omega_m}$). In particular, it is selfadjoint with domain $H^2$, its $L^2$-spectrum is discrete, contained in the positive real line and accumulates near $+\infty$. The eigenstates are smooth. By assumption, we have assumed that there are CKTs for the connection obtained with $\Gamma = 0$. We denote by $\Pi$ the $L^2$-orthogonal projection on the eigenstates at $0$ (the CKTs): it can be written as
\[
\Pi = \sum_{i=1}^d \langle \cdot , u_i \rangle_{L^2(M,\Omega_m)} u_i,
\]
where $(u_1, ...,u_d)$ forms an orthonormal family for the $L^2$ scalar product and $d$ is the dimension of the eigenspace at zero.

We choose a small (counter clockwise oriented) circle $\gamma$ around $0$ so that inside $\gamma$, $0$ is the only eigenvalue for the operator $\Delta_+^{\Gamma=0}$. Of course, this is an open property in the sense that it is still true for any small perturbation $\Delta_+^\Gamma$ with $\Gamma \neq 0$ of the operator. We introduce
\[
\Pi^\Gamma := \dfrac{1}{2\pi i} \int_{\gamma} (z-\Delta_+^\Gamma)^{-1} \dd z.
\]
For $\Gamma = 0$, we have $\Pi^{\Gamma} = \Pi$ is the $L^2$-orthogonal projection on the CKTs. For $\Gamma \neq 0$, some eigenvalues may leave $0$ (but they still have to be contained in the positive real line) and $\Pi^{\Gamma}$ is the $L^2$-orthogonal projection on all the eigenvalues contained inside the circle $\gamma$. We then define
\[
\lambda_\Gamma := \Tr\left(\Delta_+^\Gamma \Pi^\Gamma\right),
\]
which is the sum of the eigenvalues contained inside $\gamma$. Of course, for $\Gamma = 0$, $\lambda_{\Gamma=0}=0$. The map
\[
C^\infty(M,T^*M \otimes \mathrm{End}_{\mathrm{sk}}(\mc{E})) \ni \Gamma \mapsto (\lambda_{\Gamma},\Pi^\Gamma) \in \R \times\mc{L}(L^2)
\]
is smooth and we are going to compute its first and second derivatives at $\Gamma=0$. We start with the first derivative.

\begin{lemma}
\label{lemma:first-variation}
For all $A \in C^\infty(M,T^*M \otimes \mathrm{End}_{\mathrm{sk}}(\mc{E}))$, $\dd \lambda_{\Gamma=0}(A) = 0$.
\end{lemma}

\begin{proof}
We consider for small $s \in \R$ the family of operators $\Delta^{sA}_+ = -\X^{sA}_-\X^{sA}_+$. We have:
\[
\left. \dfrac{\dd}{\dd s} \Pi^{sA} \right|_{s=0} = \left. \dfrac{\dd}{\dd s} \dfrac{1}{2\pi i} \int_{\gamma} (z-\Delta_+^{sA})^{-1} \dd z  \right|_{s=0} =\dfrac{1}{2\pi i} \int_{\gamma} (z-\Delta_+)^{-1}\dot{\Delta}_+ (z-\Delta_+)^{-1} \dd z,
\]
and:
\[
\left. \dfrac{\dd}{\dd s} \lambda_{sA} \right|_{s=0}  = \Tr\left(\left. \dfrac{\dd}{\dd s} \Delta_+^{sA} \right|_{s=0}  \Pi \right) + \Tr \left(\Delta_+ \left. \dfrac{\dd}{\dd s} \Pi^{sA} \right|_{s=0} \right).
\]
We claim that both terms vanish (this is always the case for the second term, whatever the perturbation actually). Indeed, for the first term, we observe that
\[
\left. \dfrac{\dd}{\dd s} \Delta_+^{sA} \right|_{s=0} = - \X_- A_+ - A_- \X_+
\]
and thus, using that $\ker \Delta_+ = \ker \X_+$, we obtain:
\[
\begin{split}
\Tr\left(\left. \dfrac{\dd}{\dd s} \Delta_+^{sA} \right|_{s=0}  \Pi \right) & = -\Tr( (\X_- A_+ + A_- \X_+) \Pi) \\
& =- \Tr\left( \sum_i \langle \cdot,u_i \rangle_{L^2} \X_-A_+ u_i \right) \\
& = - \sum_i \langle \X_- A_+ u_i, u_i \rangle_{L^2}  = \sum_i \langle A_+ u_i, \X_+ u_i \rangle_{L^2} = 0.
\end{split}
\]
As far as the second term is concerned, we have using that $(z-\Delta_+)^{-1} = \Pi/z + R(z)$, where $R$ is holomorphic:
\[
\begin{split}
\Delta_+  \left. \dfrac{\dd}{\dd s} \Pi^{sA} \right|_{s=0} & = \dfrac{1}{2\pi i} \int_{\gamma} \Delta_+  (z-\Delta_+)^{-1}\dot{\Delta}_+ (z-\Delta_+)^{-1} \dd z  \\
& = \dfrac{1}{2\pi i} \int_{\gamma} \underbrace{(\Delta_+ -z)  (z-\Delta_+)^{-1}}_{=-\mathbbm{1}}\dot{\Delta}_+ (z-\Delta_+)^{-1} \dd z\\
&+  \dfrac{1}{2\pi i} \int_{\gamma} z  (z-\Delta_+)^{-1}\dot{\Delta}_+ (z-\Delta_+)^{-1} \dd z = - \dot{\Delta}_+ \Pi + \Pi \dot{\Delta}_+ \Pi,
\end{split}
\]
and taking the trace, we obtain:
\[
\begin{split}
\Tr\left( \Delta_+  \left. \dfrac{\dd}{\dd s} \Pi^{sA} \right|_{s=0} \right) &= \Tr\left( - \dot{\Delta}_+ \Pi + \Pi \dot{\Delta}_+ \Pi\right) \\
& = \Tr\left(- \dot{\Delta}_+ \Pi  + \dot{\Delta}_+ \Pi^2 \right)  = \Tr\left(- \dot{\Delta}_+ \Pi  + \dot{\Delta}_+ \Pi\right) = 0.
\end{split}
\]
This concludes the proof.
\end{proof}

We now compute the second variation of $\lambda$.

\begin{lemma}
\label{lemma:second-variation}
For all $A \in C^\infty(M,T^*M \otimes \mathrm{End}_{\mathrm{sk}}(\mc{E}))$:
\[
\dd^2\lambda_{\Gamma=0}(A,A) = \sum_{i=1}^d \|\pi_{\ker \X_-} A_+ u_i\|^2_{L^2}.
\]
\end{lemma}

\begin{proof}
This is a rather tedious computation. We have:
\[
\dd^2\lambda_{\Gamma=0}(A,A) = \underbrace{\Tr\left(\ddot{\Delta}_+ \Pi \right)}_{:=\text{(I)}} + 2  \underbrace{\Tr\left(\dot{\Delta}_+\dot{\Pi} \right)}_{:=\text{(II)}} +  \underbrace{\Tr\left( \Delta_+ \ddot{\Pi} \right)}_{:=\text{(III)}}.
\]
Since $\ddot{\Delta}_+ = 2(\dot{\X}_+)^* \dot{\X}_+$, the first term gives:
\[
\text{(I)} = \Tr\left(\ddot{\Delta}_+ \Pi \right) = 2 \Tr((\dot{\X}_+)^* \dot{\X}_+ \Pi) = 2\sum_i \|\dot{\X}_+ u_i\|^2_{L^2} = 2 \sum_i \|A_+ u_i\|^2_{L^2}.
\]
For the second term, we use that
\[
\dot{\Pi} = \Pi \dot{\Delta}_+ R(0) + R(0) \dot{\Delta}_+ \Pi,
\]
where we recall that $R$ is defined by $(z-\Delta_+)^{-1} = \Pi/z + R(z)$. Thus:
\[
\text{(II)} = 2\Tr(\dot{\Delta}_+ \Pi \dot{\Delta}_+ R(0)) + 2 \Tr( \dot{\Delta}_+ R(0) \dot{\Delta}_+ \Pi).
\]
Note that the first term vanishes as $\X_+\Pi = 0$ and $\X_-^* = -\X_+$. And last but not least, we compute the third term. First of all, we have:
\begin{multline*}
\ddot{\Pi}  = 2 \times \dfrac{1}{2\pi i} \int_\gamma (z-\Delta_+)^{-1}  \dot{\Delta}_+  (z-\Delta_+)^{-1} \dot{\Delta}_+  (z-\Delta_+)^{-1} \dd z\\  + \dfrac{1}{2 \pi i} \int_\gamma  (z-\Delta_+)^{-1}\ddot{\Delta}_+  (z-\Delta_+)^{-1} \dd z.
\end{multline*}
This implies after some simplification that:
\[
\begin{split}
\Delta_+ \ddot{\Pi} & = -2 \dot{\Delta}_+  \dfrac{1}{2\pi i} \int_\gamma (z-\Delta_+)^{-1}  \dot{\Delta}_+  (z-\Delta_+)^{-1} \dd z \\
&+ 2 \times  \dfrac{1}{2\pi i} \int_\gamma z (z-\Delta_+)^{-1}  \dot{\Delta}_+  (z-\Delta_+)^{-1} \dot{\Delta}_+  (z-\Delta_+)^{-1} \dd z \\
& - \ddot{\Delta}_+ \Pi \\
& +  \dfrac{1}{2\pi i} \int_\gamma z (z-\Delta_+)^{-1}  \ddot{\Delta}_+  (z-\Delta_+)^{-1} \dd z \\
& = -2(\dot{\Delta}_+\Pi\dot{\Delta}_+ R(0) + \dot{\Delta}_+ R(0) \dot{\Delta}_+ \Pi) \\
& + 2(\Pi \dot{\Delta}_+ \Pi \dot{\Delta}_+ R(0) + \Pi \dot{\Delta}_+ R(0) \dot{\Delta}_+ \Pi + R(0) \dot{\Delta}_+ \Pi \dot{\Delta}_+ \Pi) \\
&  -\ddot{\Delta}_+ \Pi \\
&  + \Pi \ddot{\Delta}_+ \Pi.
\end{split}
\]
It is an elementary computation, using that $\X_+ \Pi = 0$ and $\Pi \X_- = 0$, to show that $\Pi \dot{\Delta}_+ \Pi = 0$. Thus:
\[
\Delta_+ \ddot{\Pi} = -2(\dot{\Delta}_+\Pi\dot{\Delta}_+ R(0) + \dot{\Delta}_+ R(0) \dot{\Delta}_+ \Pi) + 2 \Pi \dot{\Delta}_+ R(0) \dot{\Delta}_+ \Pi -\ddot{\Delta}_+ \Pi + \Pi \ddot{\Delta}_+ \Pi.
\]
Taking the trace, using that $\Pi^2 = \Pi$, we get:
\[
\Tr(\Delta_+ \ddot{\Pi}) = -2 \Tr(\dot{\Delta}_+\Pi\dot{\Delta}_+ R(0)).
\]
Summing the contributions $\text{(I,\, II,\, III)}$, we obtain:
\[
\left. \dfrac{\dd^2}{\dd s^2} \lambda_{sA} \right|_{s=0} =2 \sum_i \|\dot{\X}_+ u_i\|^2_{L^2}  + 2\Tr( \dot{\Delta}_+ R(0) \dot{\Delta}_+ \Pi).
\]
It remains to study this last term. After some computations, one can show that:
\[
\Tr( \dot{\Delta}_+ R(0) \dot{\Delta}_+ \Pi) = \sum_{i=1}^d \langle \dot{\X}_+ u_i, \X_+ R(0) \X_+^* \dot{\X}_+ u_i \rangle.
\]
We now study $Q := - \X_+ R(0) \X_+^*$. Recall that any $f \in C^{\infty}(M,\Omega_{m+1} \otimes \E)$ can be uniquely decomposed as $f = \X_+u + h$, where $h \in \ker \X_-$ and $u \in \ker \Pi$. Applying $(\X_+)^*$, we get $(\X_+)^*f =-\X_-f = (\X_+)^*\X_+u + 0 = \Delta_+ u$. Using the equality $\mathbbm{1}-\Pi = -R(0)\Delta_+$, and the fact that $\X_+ \Pi = 0$, we then obtain that:
\[
\X_+ u = -\X_+ R(0) (\X_+)^*f = Qf.
\]
In other words, $Q = \pi_{\ran(\X_+)}$ is the $L^2$-orthogonal projection on $\ran(\X_+)$. Thus:
\[
\begin{split}
\left. \dfrac{\dd^2}{\dd s^2} \lambda_{sA} \right|_{s=0} & = 2 \sum_i \|\dot{\X}_+ u_i\|^2_{L^2}  - \langle\dot{\X}_+ u_i,\pi_{\ran(\X_+)} \dot{\X}_+ u_i \rangle_{L^2} \\
& = 2 \sum_i \|\pi_{\ker \X_-} \dot{\X}_+ u_i\|^2_{L^2} = 2 \sum_i \|\pi_{\ker \X_-} A_+ u_i\|^2_{L^2}.
\end{split}
\]
\end{proof}

\subsection{Proof of the generic absence of CKTs}

\label{ssection:absence-ckt}

As mentioned in the introduction, Theorem \ref{theorem:main} follows from the following result.

\begin{theorem}
\label{theorem:ckts}
Let $(\mc{E},\nabla^{\mc{E}})$ be Hermitian vector bundle over the Riemannian manifold $(M,g)$, equipped with a smooth unitary connection. Assume that $\ker (\X_+|_{C^\infty(M,\Omega_m \otimes \E)})$ is not trivial. Then, for all $k \geq 2$ and $\eps > 0$, there exists a unitary ${\nabla'}^{\mc{E}}$ such that $\|\nabla^{\mc{E}}-{\nabla'}^{\mc{E}}\|_{C^k(M,T^*M \otimes \mathrm{End}(\mc{E}))} < \eps$ and $(\mc{E},{\nabla'}^{\mc{E}})$ has no twisted CKTs of degree $m$.
\end{theorem}

Note that the definition of the $C^k$ norms (which may depend on some choice of coordinates) is irrelevant.

\begin{proof}
Assume that $\ker \X_+|_{C^\infty(M,\Omega_m \otimes \E)}$ is $d$-dimensional. Consider a small circle around $0$ in $\C$ in which $0$ is the only eigenvalue (with multiplicity $d$). It is sufficient to prove that we can produce an arbitrary small perturbation $\nabla^{\mc{E}}+\Gamma$ such that the sum of the eigenvalues inside the circle is strictly positive. Since the Laplacians $\Delta_+^\Gamma \geq 0$ are self-adjoint and non-negative, this means that at least one of the eigenvalues at $0$ was ejected, namely $\ker  \X^\Gamma_+|_{C^\infty(M,\Omega_m \otimes \E)}$ is at most $d-1$ dimensional. Then, repeating the process finitely many times, one can eject all the resonances out of $0$, i.e. one obtains a connection $\nabla^{\mc{E}}+\Gamma$, where $\Gamma$ is arbitrarily small (in $C^k$) such that $\X^\Gamma_+$ has no eigenvalues at $0$.

Now, using Lemma \ref{lemma:second-variation}, in order to produce a perturbation $\nabla^{\mc{E}}+\Gamma$ of the connection $\nabla^{\mc{E}}$ such that the sum of the eigenvalues of $\Delta^{\Gamma}_+$ inside the circle is strictly positive, it is sufficient to take $\nabla^{\mc{E}}+sA$ (where $A$ is $C^k$), for $s$ small enough and where $\pi_{\ker \X_-} A_+ u_i \neq 0$ (for some $i \in \left\{1,...,d\right\}$). Therefore, this boils down to the following result:

\begin{lemma}\label{lemma:Gammaexistence}
Assume $u_0 \in \ker (\X_+|_{C^\infty(M,\Omega_m \otimes \E)})$. Then, there exists $\Gamma \in C^\infty(M,T^*M \otimes \mathrm{End}_{\sk}(\mc{E}))$ such that $\pi_{\ker \X_-} \Gamma_+ u_0 \neq 0$.
\end{lemma}
\begin{proof}
Assume this is not the case. Then, using the splitting
\[
C^\infty(M,\Omega_{m+1} \otimes \mc{E}) = \ker \X_-|_{C^\infty(M,\Omega_{m+1}\otimes \mc{E})} \oplus^\bot \X_+(C^\infty(M,\Omega_m\otimes \mc{E})),
\]
we obtain that for all $\Gamma \in C^\infty(M,T^*M \otimes \mathrm{End}_{\sk}(\mc{E}))$, there exists $f_\Gamma \in C^\infty(M,\Omega_m \otimes \mc{E})$ such that $\Gamma_+ u_0 = \X_+ f_\Gamma$. Thus, for all $w \in \ker \X_-|_{C^\infty(M,\Omega_{m+1} \otimes \mc{E})}$:
\[
\langle \Gamma_+ u_0, w \rangle_{L^2} = \langle \X_+ f_\Gamma,w \rangle_{L^2} = - \langle f_\Gamma, \X_- w \rangle_{L^2} = 0 = \langle u_0, \Gamma_- w \rangle_{L^2}.
\]
We claim that this implies that pointwise in $x \in M$, one has $0 = \langle u_0(x), \Gamma_-(x) w(x) \rangle_{\Omega_m(x) \otimes \mc{E}_x}$, for all $\Gamma\in C^\infty(M,T^*M \otimes \mathrm{End}_{\sk}(\mc{E}))$ and for all $w \in \ker \X_-|_{C^\infty(M,\Omega_{m+1} \otimes \mc{E})}$. Assuming the claim, we then use that $\X_-$ is of uniform divergence type (as proved in Lemma \ref{lemma:uniform-x-}): by Lemma \ref{lemma:microlocal-surjectivity}, it implies that the map 
\[
\ker \X_-|_{C^\infty(M,\Omega_{m+1}\otimes \mc{E})} \ni w \mapsto w(x) \in \Omega_{m+1}(x) \otimes \mc{E}_x
\]
is surjective for all $x \in M$. We then apply Lemma \ref{lemma:algebraic-surjectivity} which allows to find $\Gamma$ and $w$ such that $\Gamma_-(x)w(x) = u_0(x)$. Thus $\langle u_0(x), u_0(x) \rangle = 0$ for all $x \in M$ that is $u_0 = 0$. This is a contradiction.

It now remains to prove that pointwise in $x \in M$, we have $0 = \langle u_0(x), \Gamma_-(x) w(x) \rangle_{\Omega_m(x) \otimes \mc{E}_x}$. We fix $x_0 \in M$ and consider an arbitrary $w \in \ker \X_-|_{C^\infty(M,\Omega_{m+1} \otimes \mc{E})}, \Gamma \in C^\infty(M,T^*M \otimes \mathrm{End}_{\sk}(\mc{E}))$. We consider a sequence of (real-valued) functions $\varphi_h \in C^\infty(M)$ such that $\varphi_h \rightarrow_{h \rightarrow 0} \delta_{x_0} \in \mc{D}'(M)$, where the convergence takes place in the sense of distributions, i.e. we have $\langle \varphi_h, f \rangle_{L^2(M)} \rightarrow_{h \rightarrow 0} f(x_0)$, for all $f \in C^\infty(M)$. Then:
\[
\begin{split}
0 = \langle u_0, (\varphi_h \Gamma)_- w \rangle_{L^2} & = \int_{M} \varphi_h(x) \langle u_0(x), \Gamma_-(x) w(x) \rangle_{\Omega_m(x) \otimes \mc{E}_x} \dd \vol_g(x)\\
&  \rightarrow_{h \rightarrow 0} \langle u_0(x_0), \Gamma_-(x_0) w(x_0) \rangle_{\Omega_m(x_0) \otimes \mc{E}_{x_0}}.
\end{split}
\]

\end{proof}

This concludes the proof of the main Theorem.
\end{proof}

In the case of surfaces the operator $\mathbf{X}_-$ is not of uniformly divergent type, but we may still perturb the CKTs in some cases. Let $(\E,\nabla^{\E})$ and $(M,g)$ be as in Theorem \ref{theorem:main} and assume $\dim M = 2$ with $M$ orientable of genus $\mathbf{g}$. As $M$ admits a complex structure, we may consider $(T^*M)^{0, 1} =: \mathcal{K}$ the canonical bundle spanned locally by the forms $dz$ and analogously $\mathcal{K}^{-1} := (T^*M)^{1, 0}$ locally spanned by $d\bar{z}$. One checks that $\otimes^{m}_ST^*M|_{0-\Tr} = \mathcal{K}^{\otimes m} \oplus \mathcal{K}^{\otimes (-m)}$ and using the map $\pi^*_m$ there is a splitting for each $m \neq 0$: $\Omega_m = H_m \oplus H_{-m}$. We write $\Omega_0 = H_0$. The operators $\mathbf{X}_{\pm}$ for $m > 0$ decompose as $\mathbf{X}_{\pm}|_{\Omega_m \otimes \E} = \mu_+ \oplus \mu_-$, where $\mu_{\pm}: H_m \otimes \E \to H_{m \pm 1} \otimes \E$ for any $m$ (see e.g. \cite{Paternain-09} for details). Generalising our earlier approach, we obtain for surfaces:

\begin{proposition}\label{prop:n=2}
	Let $m > 0$. If $\ker \mathbf{X}_-|_{\Omega_{m + 1} \otimes \E} \neq \{0\}$, there is a perturbation that reduces the dimension of $\ker \mathbf{X}_+|_{\Omega_m \otimes \E}$ by at least one. Consequently, if the index $\ind \mu_+|_{H_{m} \otimes \E} \leq 0$, then it is possible to perturb the connection to eject all the twisted CKTs; if $\ind \mu_+|_{H_{m} \otimes \E} > 0$, then if necessary we may perturb the connection to obtain
	\[\dim \ker \mu_+|_{H_m \otimes \E} = \ind \mu_+|_{H_{m} \otimes \E}.\]
\end{proposition}

\begin{proof}
	We first consider an equivalent of Lemma \ref{lemma:algebraic-surjectivity} for the case of surfaces. In local isothermal coordinates $g = e^{2\lambda} |dz|^2$, we may write the connection form as $\Gamma = \Gamma(\partial_z) dz + \Gamma(\partial_{\bar{z}}) d\bar{z}$; set $\Gamma_- = \Gamma(\partial_{\bar{z}}) d\bar{z}$. Then $\Gamma_-: H_m \otimes \E \to H_{m-1} \otimes \E$ is given by 
	\[\Gamma_- (e^{m \lambda} (dz)^m \otimes s) = e^{(m-1)\lambda} (dz)^{m-1} \otimes \Gamma(e^{-\lambda} \partial_{\bar{z}}) s.\]
	Therefore as soon as $\Gamma(\partial_{\bar{z}})(x): \E_x \to \E_x$ is invertible, we have $\Gamma_-(x) : H_m(x) \otimes \E_x \to H_{m-1}(x) \otimes \E_x$ an isomorphism.
	
	Coming back to the main proof, we may without loss of generality assume that $u_0 \in C^\infty(M; H_m \otimes \E)$. Arguing by contradiction as in the proofs of Theorem \ref{theorem:ckts} and Lemma \ref{lemma:Gammaexistence} shows $\langle{u_0(x), \Gamma_-(x)w(x)}\rangle_{\Omega_m(x) \otimes \E_x} = 0$ for all $x\in M$, $\Gamma \in C^\infty(M; T^*M \otimes \End_{\sk} \E)$ and $w \in \ker \mu_-|_{H_{m + 1} \otimes \E}$. Assume $0 \neq w \in \ker \mu_-|_{H_{m + 1} \otimes \E}$. Since $\mu_-$ is elliptic, $\{w = 0\} \subset M$ is nowhere dense by the unique continuation principle. Picking a suitable $\Gamma$ according to the previous paragraph, we obtain $u_0 = 0$ on $\supp (w)$, thus $u_0 \equiv 0$, contradiction. Therefore to reduce the dimension of $\ker \mathbf{X}_+|_{\Omega_m \otimes \E}$ by at least one, it suffices to produce a single non-trivial element in $\ker \mu_-|_{H_{m + 1} \otimes \E}$, proving the first part of the claim.
	
	If $\eta_\pm$ denote the raising/lowering operators for $\E = \mathbb{C}$, then by \eqref{eq:D*symbol} and \eqref{equation:link-xd} $\sigma_{\mu_\pm}(x, \xi) = \sigma_{\eta_\pm}(x, \xi) \otimes \id_{\E}$, so the value of $\ind \mu_-|_{H_{m+1} \otimes \E}$ is topological.\footnote{By the Atiyah-Singer index theorem, it may be computed explicitly as a function of the first Chern class $c_1(\E)$ and $\ind \eta_-|_{H_{m+1}} = (2 m  + 1)(\mathbf{g} - 1)$ (for the latter see \cite[Lemma 2.1]{Paternain-Salo-Uhlmann-14-2}).} Note that $\mu_+^* = -\mu_-$ implies $\ind \mu_+|_{H_m \otimes \E} = - \ind \mu_-|_{H_{m + 1} \otimes \E}$. Thus by the previous paragraph, as long as $\ker \mu_-|_{H_{m + 1}\otimes \E} \neq \{0\}$, an inductive argument ejecting the eigenvalues one by one shows that we may reduce the dimension of $\ker \mu_+|_{H_m \otimes \E}$ to $\max\{0, \ind \mu_+|_{H_m \otimes \E}\}$, completing the proof.

\end{proof}

\subsection{Generic absence of CKTs on the endomorphism bundle}

\label{ssection:absence-ckts-endomorphism}

From now on, we assume that $(M,g)$ has no non-trivial CKTs on its trivial line bundle. Given a unitary connection $\nabla^{\E}$ on $\E \rightarrow M$, recall that it induces a canonical unitary connection $\nabla^{\mathrm{End}(\E)}$ on $\mathrm{End}(\mc{E})$, defined by the following (Leibniz) identity, holding for all $u \in C^\infty(\M,\mathrm{End}(\mc{E})), f \in C^\infty(M,\mc{E})$:
\begin{equation}
\label{equation:connection-end}
\nabla^{\mc{E}}\left(u(f)\right) = (\nabla^{\mathrm{End}(\E)}u)(f) + u(\nabla^{\mc{E}}f).
\end{equation}
In particular, if ${\nabla'}^{\mc{E}} = \nabla^{\mc{E}} + \Gamma$ (with $\Gamma \in C^\infty(\M,\mathrm{End}(\mc{E}))$) is another connection on $\mc{E}$, then one has the well-known formula: 
\begin{equation}
\label{equation:diff-connection-end}
{\nabla'}^{\mathrm{End}(\E)} = \nabla^{\mathrm{End}(\E)} + [\Gamma, \cdot].
\end{equation}
In this subsection, we study the absence of CKTs on the endomorphism bundle $\mathrm{End}(\E)$ equipped with $\nabla^{\mathrm{End}(\E)}$ and prove Theorem \ref{theorem:generic-ckts-endomorphism}, namely that there are generically no non-trivial CKTs. (Observe from the definition that $\nabla^{\mathrm{End}(\E)} \mathbbm{1}_{\mc{E}} = 0$ and this is a CKT of order $0$, so unlike the general case of $\nabla^{\mc{E}}$ acting on sections of $\mc{E}$, we will never be able to exclude all the CKTs of order $0$ on the endomorphism bundle.) 

\subsubsection{Preliminary discussion}

We need to clarify notations here. Given an element $u \in C^\infty(M,\otimes^m_S T^*M \otimes \mathrm{End}(\E))$, there are two notions of traces which appear, the one on symmetric tensors and the one on the endomorphism part. Consider a fixed point $x_0 \in M$ and let $e_1,...,e_r$ be an orthonormal frame for $\mc{E}$ in a neighborhood of $x_0$. Then $(e_i^* \otimes e_j)_{i,j=1}^r$ is a (local) orthonormal frame for $\mathrm{End}(\E)$ and we can write $u = \sum_{i,j=1}^r u_{ij} \otimes (e_i^* \otimes e_j)$, where $u_{ij} \in C^\infty(M,\otimes^m_S T^*M)$. We introduce the trace on the symmetric part
\[
\Tr_{\mathrm{Sym}} : C^\infty(M, \otimes^{m}_S T^*M \otimes  \mathrm{End}(\E)) \rightarrow C^\infty(M, \otimes^{m-2}_S T^*M \otimes  \mathrm{End}(\E)),
\]
defined by the following equality:
\begin{equation}
\label{equation:trace-symmetric}
\Tr_{\mathrm{Sym}}(u) :=  \sum_{i,j=1}^r \mc{T}(u_{ij}) \otimes (e_i^* \otimes e_j) \in C^\infty(M, \otimes^{m-2}_S T^*M \otimes  \mathrm{End}(\E)),
\end{equation}
where the trace $\mc{T}$ on the right-hand side is the one introduced in \eqref{equation:trace-tensor}. We also introduce the trace on the endomorphism as a map
\[
\Tr_{\mathrm{End}} : C^\infty(M, \otimes^{m}_S T^*M \otimes  \mathrm{End}(\E)) \rightarrow C^\infty(M, \otimes^{m}_S T^*M),
\]
defined by:
\begin{equation}
\label{equation:trace-endomorphism}
\Tr_{\mathrm{End}} (u) = \sum_{i,j=1}^r u_{ij} \otimes \Tr(e_i^* \otimes e_j) = \sum_{i=1}^r u_{ii}.
\end{equation}

In the following, we will keep the notation $\nabla^{\mathrm{End}(\E)}$ for the connection on $\mathrm{End}(\E)$ and, in order to be consistent with \S\ref{ssection:twisted-tensors}, we will write $D_{\mathrm{End}(\E)}$ for the symmetric derivatives acting as
\[
D_{\mathrm{End}(\E)} : C^\infty(M,\otimes^m_S T^*M \otimes \mathrm{End}(\E)) \rightarrow C^\infty(M,\otimes^{m+1}_S T^*M \otimes \mathrm{End}(\E)),
\]
as introduced in \eqref{equation:formula-de}. Also recall that $\mc{P}$ denotes the orthogonal projection onto trace-free symmetric tensors. We then have the following lemma:

\begin{lemma}
\label{lemma:commutation-trace}
For all $u \in C^\infty(M,\otimes^m_S T^*M|_{0-\Tr} \otimes \mathrm{End}(\E))$, we have:
\[
\Tr_{\mathrm{End}}\mc{P} D_{\mathrm{End}(\E)} u =  \mc{P} D \Tr_{\mathrm{End}} u.
\]
\end{lemma}

Note that there is an abuse of notations in the previous equality insofar as the operator $\mc{P}$ appearing on both sides is not exactly the same.

\begin{proof}
First of all, it is straightforward that $\Tr_{\mathrm{End}}\mc{P} = \mc{P} \Tr_{\mathrm{End}}$. Then, observe that the left-hand side in Lemma \ref{lemma:commutation-trace} is independent of any choice of connection. Indeed, for the connection $\nabla^{\mc{E}} + \Gamma$, one obtains
\[
\Tr_{\mathrm{End}}\mc{P} (D_{\mathrm{End}(\E)} + [\Gamma,\cdot]) = \mc{P} \Tr_{\mathrm{End}} (D_{\mathrm{End}(\E)} + [\Gamma,\cdot]) =  \mc{P} \Tr_{\mathrm{End}} D_{\mathrm{End}(\E)} + \mc{P} \Tr_{\mathrm{End}}  [\Gamma,\cdot],
\]
and it is straightforward to check that $\Tr_{\mathrm{End}}  [\Gamma,\cdot] = 0$. As a consequence, it is sufficient to prove the statement in local coordinates for the trivial flat connection on the trivial vector bundle $\C^r$. Once again, this is immediate.
\end{proof}

We now introduce the map
\[
\mathrm{Ad} : C^\infty(M,\otimes^m_S T^*M \otimes \mathrm{End}(\E))  \rightarrow C^\infty(M,\otimes^m_S T^*M \otimes \mathrm{End}(\E)),
\]
defined by taking the pointwise adjoint (with respect to the metric $h$ on $\mc{E}$) on the endomorphism part. We have:

\begin{lemma}
\label{lemma:commutation-adjoint}
For all $u \in C^\infty(M,\otimes^m_S T^*M|_{0-\Tr} \otimes \mathrm{End}(\E))$, we have:
\[
\mathrm{Ad}( \mc{P} D_{\mathrm{End}(\E)} u ) =   \mc{P} D_{\mathrm{End}(\E)} \mathrm{Ad} (u).
\]
\end{lemma}

\begin{proof}
It is immediate that $\mathrm{Ad} ~\mc{P} = \mc{P} \mathrm{Ad}$. Moreover, given $\Gamma \in C^\infty(M,T^*M \otimes \mathrm{End}_{\mathrm{sk}}(\E))$, one has $\mathrm{Ad}[\Gamma,\cdot] = [\Gamma,\cdot] \mathrm{Ad}$ (here, the fact that $\Gamma$ is skew-hermitian is crucial). Therefore, similarly to Lemma \ref{lemma:commutation-trace}, the statement boils down in local coordinates to the trivial flat connection on the trivial bundle $\C^r$ and this is immediate.
\end{proof}

\subsubsection{Proof of Theorem \ref{theorem:generic-ckts-endomorphism}}

\label{sssection:absence-ckts-endomorphism}

The proof relies exactly on the same strategy as in \S\ref{ssection:absence-ckt}. We assume that $\nabla^{\mc{E}}$ is a connection such that $\nabla^{\mathrm{End}(\E)}$ has a non-trivial CKT $u$ of order $m \in \N$. We then consider the pullback operator $\X := (\pi^* \nabla^{\mathrm{End}(\E)})_X$ on $SM$, acting on sections of $\mathrm{End}(\pi^* \mc{E})$. If we perturb $\nabla^{\mc{E}}$ by $\nabla^{\mc{E}} + \Gamma$ (with $\Gamma \in C^\infty(M,T^*M \otimes \mathrm{End}_{\mathrm{sk}}(\E))$), we obtain $\X^\Gamma := (\pi^* (\nabla^{\mathrm{End}(\E)}+[\pi^*\Gamma,\cdot]))_X$ which can still be decomposed as $\X^\Gamma = \X^\Gamma_+ + \X^\Gamma_-$, where as before
\[
\X^\Gamma_{\pm} : C^\infty(M,\Omega_m \otimes \mathrm{End}(\E)) \rightarrow C^\infty(M,\Omega_{m\pm 1} \otimes \mathrm{End}(\E)).
\]
Writing $\X_+ := \X_+^{\Gamma=0}$, we have by assumption $\X_+ u = 0$. We introduce $\Delta_+^\Gamma := -\X_-^\Gamma \X_+^\Gamma \geq 0$, $\Pi^\Gamma$ is defined as before by
\[
\Pi^\Gamma := \dfrac{1}{2\pi i} \int_\Gamma (z-\Delta^\Gamma_+)^{-1} \dd z,
\]
for some small counter clockwise oriented circle $\gamma$ around $0$ and $\lambda_\Gamma := \Tr(\Delta_+^\Gamma \Pi^\Gamma)$. Our goal is to find a small non-zero $\Gamma$ such that $\lambda_\Gamma > 0$. Repeating the arguments of \S\ref{ssection:absence-ckt}, we then conclude that there are generically no CKTs on the endomorphism bundle. As far as the first and second perturbation of $\lambda_\Gamma$ at $\Gamma=0$ are concerned, we obtain:
\[
\dd \lambda_{\Gamma=0}(A) = 0,\,\,\, \dd^2 \lambda_{\Gamma=0}(A,A) = \sum_{i=1}^d \|\pi_{\ker \X_-} [A,\cdot]_+ u_i\|^2_{L^2},
\]
where $(u_1,...,u_d)$ form a $L^2$-orthonormal basis of CKTs of degree $m$. (Here, recall that
\[
[A,\cdot] : C^\infty(M,\Omega_m \otimes \mathrm{End}(\E)) \rightarrow C^\infty(M,\Omega_{m-1} \otimes \mathrm{End}(\E) \oplus \Omega_{m+1} \otimes \mathrm{End}(\E)),
\]
so there are natural positive $[A,\cdot]_+$ and negative $[A,\cdot]_-$ operators defined.)

\begin{proof}[Proof of Theorem \ref{theorem:generic-ckts-endomorphism}]
First of all, using Lemma \ref{lemma:commutation-trace}, we have $\Tr_{\mathrm{End}} \mc{P} D_{\mathrm{End}(\E)} u = 0 = \mc{P} D \Tr_{\mathrm{End}} u$. Since $(M,g)$ is assumed not to have nontrivial CKTs on its trivial line bundle, we obtain the following: if $m > 0$, $\Tr_{\mathrm{End}} u = 0$; if $m=0$, $\Tr_{\mathrm{End}} u  = c \in \C$ is constant. In the case where $m = 0$, we can further always assume that
\[
\langle u, \mathbbm{1}_{\mc{E}} \rangle_{L^2} = \int_M \Tr(u) \dd \vol_g = 0 = \vol(M) \times c,
\]
that is $c=0$, insofar as we do not perturb the resonant state $\mathbbm{1}_{\mc{E}}$. As a consequence, we can always take $u$ such that $\Tr_{\mathrm{End}}(u) = 0$. Moreover, we can always assume that $\mathrm{Ad}(u) = -u$. Indeed, if not, then we can take $u' := u-\mathrm{Ad}(u)$ which satisfies $\mathrm{Ad}(u') = -u'$, and is also a CKT of order $m$ by Lemma \ref{lemma:commutation-adjoint}. If $u' \neq 0$, this is fine. If $u'=0$, then taking $i \times u$, we have the right candidate. To sum up this preliminary discussion, we take a non-zero $u \in C^\infty(M,\otimes^m_S T^*M|_{0-\Tr} \otimes \mathrm{End}(\E))$ such that $\mc{P} D_{\mathrm{End}}(u) = 0, \Tr_{\mathrm{End}}(u) = 0$ and $\mathrm{Ad}(u) = -u$. Our goal is to show that there exists $A \in C^\infty(M,T^*M \otimes \mathrm{End}_{\mathrm{sk}}(\E))$ such that $\pi_{\ker \X_-} [A,\cdot]_+ u \neq 0$. We argue by contradiction and assume this is not the case. Then, following \S\ref{ssection:absence-ckt}, this implies that for all $w \in \ker \X_-|_{C^\infty(M,\Omega_{m+1} \otimes \mathrm{End}(\E))}$, one has:
\[
\langle u, [A,\cdot]_- w \rangle_{L^2} = 0.
\]
Of course, taking as in \S\ref{ssection:absence-ckt} a sequence $\varphi_h \rightarrow \delta_{x_0}$ in $\mc{D}'(M)$, we obtain:
\[
\begin{split}
0 & = \langle u, [\varphi_h A,\cdot]_- w \rangle_{L^2} \\
& = \int_{M} \varphi_h \langle u(x), [A,\cdot]_- w (x) \rangle_x \dd \vol(x) \rightarrow \langle u(x_0), [A(x_0),\cdot]_- w(x_0) \rangle,
\end{split}
\]
for all $A \in C^\infty(M,T^*M \otimes \mathrm{End}_{\mathrm{sk}}(\E))$ and $w \in \ker \X_-|_{C^\infty(M,\Omega_{m+1} \otimes \mathrm{End}(\E))}$. Since we already know that $\X_-$ is of uniform divergence type, the evaluation map
\[
\mathrm{ev}_x :  \ker \X_-|_{C^\infty(M,\Omega_{m+1} \otimes \mathrm{End}(\E))} \ni w \mapsto w(x) \in \Omega_{m+1}(T_xM) \otimes \mathrm{End}(\E_x)
\]
is surjective. In other words, we obtain that for every $x \in M$, for every $w \in \Omega_{m+1}(x) \otimes \mathrm{End}(\E_x), A \in T^*_xM \otimes \mathrm{End}_{\mathrm{sk}}(\E_x)$,
\[
\langle u(x), [A(x),\cdot]_- w(x) \rangle_x = 0.
\]
We want to show that this implies $u = 0$. This is now a purely algebraic statement since it is pointwise in $x$. We can therefore forget about $M$ and replace $T_xM$ by $E$ (a Euclidean space) and $\E_x$ by $\E$ (an inner product space). As explained in \S\ref{sssection:spherical-harmonics}, the map $r_m$ allows to identify $\mathbf{H}_m(E)$ with $\Omega_m(E)$ and we will freely use this identification from now on.

\begin{lemma}
Let $u \in \mathbf{H}_m(E) \otimes \mathrm{End}_{\mathrm{sk}}(\E)$ such that $\Tr_{\mathrm{End}}(u)=0$ and assume that $\langle u, [A,\cdot]_- w \rangle = 0$, for all $A \in E^* \otimes \mathrm{End}_{\mathrm{sk}}(\E)$, $w \in \mathbf{H}_{m+1}(E) \otimes \mathrm{End}_{\mathrm{sk}}(\E)$. Then $u=0$.
\end{lemma}

\begin{proof}
Let 
\[
P_A := [A,\cdot] : \mathbf{H}_{m}(E) \otimes \mathrm{End}(\E) \rightarrow  \mathbf{H}_{m-1}(E) \otimes \mathrm{End}(\E) \oplus  \mathbf{H}_{m+1}(E) \otimes \mathrm{End}(\E).
\]
We introduce $P_A^{\pm} := [A,\cdot]_{\pm}$ as the projections on the higher ($+$) and lower ($-$) parts. We consider an orthonormal basis $\e_1, ...,\e_n$ of $E$ and for $f \in C^\infty(E)$, we write $\partial_j f (v):= \partial_t f(v+t\e_j)|_{t=0}$.  Given $f \in \mathbf{H}_m(E)$, observe that $\partial_j f \in \mathbf{H}_{m-1}(E)$ (it is a harmonic polynomial of degree $m-1$).

Let $w = f \otimes \widetilde{w} $, where $f \in \Omega_{m+1}(E), \widetilde{w} \in \mathrm{End}(\E)$, and write $A = \sum_{j=1}^n A_j \otimes \e_j^*$ with $A_j \in \mathrm{End}_{\mathrm{sk}}(\E)$. Then:
\[
P_A^- w = \dfrac{1}{2(m-1)+n} \sum_{j=1}^n \partial_j f \otimes [A_j, \widetilde{w}]    \in C^\infty(M, \Omega_{m+1} \otimes \mathrm{End}(\E)).
\]
Indeed, by the proof of Lemma \ref{lemma:multiplication-1-form}, $P_A w (v) = A(v)w(v)-w(v)A(v) = a(v) + |v|^2 b(v)$, for some $a \in \mathbf{H}_{m+1}(E) \otimes \mathrm{End}(\E), b \in \mathbf{H}_{m-1}(E) \otimes \mathrm{End}(\E)$, with
\[
\begin{split}
&(2(m-1)+n)b(v)  =  \nabla A(v) \cdot \nabla w (v) - \nabla w (v) \cdot \nabla A(v) \\
& = \left(\sum_{j=1}^n A_j \otimes \e_j\right) \cdot \left(\sum_{j=1}^n \widetilde{w} \otimes \partial_j f(v) \e_j \right) -  \left(\sum_{j=1}^n \widetilde{w} \otimes \partial_j f(v) \e_j \right) \cdot \left(\sum_{j=1}^n A_j \otimes \e_j\right) \\
& = \sum_{j=1}^n \partial_j f(v) [A_j,\widetilde{w}].
\end{split}
\]

Now, we consider an orthonormal basis $(s_1,...,s_{r^2-1})$ of the \emph{real} vector space $\mathrm{End}_{\mathrm{sk}}(\E) \cap \ker \Tr$ (endowed with the standard scalar product on $\mathrm{End}(\E)$). We can write $u = \sum_{i=1}^{r^2-1} p_{i} \otimes s_i$, where $p_{i} \in \mathbf{H}_m(E)$. Fix some index $1 \leq i \leq r^2-1$. We consider $A = A_1 \otimes \e_1^*$, $w = f \otimes \widetilde{w}$ such that $A_1, \widetilde{w} \in \mathrm{End}_{\mathrm{sk}}(\E)$ and $[A_1,\widetilde{w}] = s_i$ (this is always possible by Lemma \ref{lemma:baby}), and $f \in \mathbf{H}_{m+1}(E)$ is such that $\partial_1 f = (2(m-1)+n) p_i$ (this is always possible by the proof of Lemma \ref{lemma:algebraic-surjectivity-trivial-bundle} since $\partial_1$ is surjective). Then:
\[
\begin{split}
\langle u, P_A^- w \rangle = 0 & = \sum_{j=1}^{r^2-1} \langle  p_{j} \otimes s_j, \dfrac{1}{2(m-1)+n} \partial_1 f(v) \otimes [A_1,\widetilde{w}] \rangle \\
& = \sum_{j=1}^{r^2-1} \langle  p_{j} \otimes s_j,  p_i \otimes s_i\rangle = \|p_i\|^2 \underbrace{\|s_i\|^2}_{=1} = \|p_i\|^2.
\end{split}
\]
Thus $u=0$.
\end{proof}
\noindent This completes the proof of Theorem \ref{theorem:generic-ckts-endomorphism}.

\end{proof}

\section{Opaque connections are generic}

\subsection{Pollicott-Ruelle resonances}

\label{ssection:resonances}

In this first paragraph, we consider the general case of a smooth closed manifold $\M$ endowed with an Anosov vector field $X$ preserving a smooth measure $\dd \mu$ and generating a flow $(\varphi_t)_{t \in \R}$. Recall that the Anosov property means that there exists a continuous flow-invariant splitting of the tangent bundle
\[
T\M = \R X \oplus E_s \oplus E_u,
\]
such that:
\begin{equation}
\label{equation:anosov}
\begin{array}{l}
\forall t \geq 0, \forall v \in E_s, ~~ \|\dd\varphi_t(v)\| \leq Ce^{-t\lambda}\|v\|, \\
\forall t \leq 0, \forall v \in E_u, ~~ \|\dd\varphi_t(v)\| \leq Ce^{-|t|\lambda}\|v\|,
\end{array}
\end{equation}
where the constants $C,\lambda > 0$ are uniform and the metric $\|\cdot\|$ is arbitrary. (It will be applied with $\M = SM$ and the geodesic vector field $X$, and assuming $X$ is Anosov.) We also introduce the dual decomposition
\[
T^*\M = \R E_0^* \oplus E_s^* \oplus E_u^*,
\]
where $E_0^*(E_s \oplus E_u) = 0, E_s^*(E_s \oplus \R X) = 0, E_u^*(E_u \oplus \R X) = 0$. We assume that $\mc{E} \rightarrow \M$ is a Hermitian vector bundle over $\M$. Let $\nabla^{\mc{E}}$ be a unitary connection on $\mc{E}$ and set $\X := \nabla^{\mc{E}}_X$. Since $X$ preserves $\dd \mu$ and $\nabla^{\mc{E}}$ is unitary, the operator $\X$ is skew-adjoint on $L^2(SM,\mc{E};\dd\mu)$, with dense domain
\begin{equation}
\label{equation:domaine-p}
\mc{D}_{L^2} := \left\{ u \in L^2(\M,\mc{E};\dd\mu) ~|~ \X u \in L^2(\M,\mc{E};\dd\mu)\right\}.
\end{equation}
Its $L^2$-spectrum consists of absolutely continuous spectrum on $i\R$ and of embedded eigenvalues. We introduce the resolvents
\begin{equation}
\label{equation:resolvent}
\begin{split}
&\RR_+(z) := (-\X-z)^{-1} = - \int_0^{+\infty} e^{-t z} e^{-t\X} \dd t, \\
&\RR_-(z) := (\X-z)^{-1} = - \int_{-\infty}^0 e^{z t} e^{-t\X} \dd t,
\end{split}
\end{equation}
initially defined for $\Re(z) > 0$. (Let us stress on the conventions here: $-\X$ is associated to the positive resolvent $\RR_+(z)$ whereas $\X$ is associated to the negative one $\RR_-(z)$.) Here $e^{-t\X}$ denotes the propagator of $\X$, namely the parallel transport by $\nabla^{\mc{E}}$ along the flowlines of $X$. Recall that for $x \in \M, t \in \R$, $C(x,t) : \mc{E}_x \rightarrow \E_{\varphi_t(x)}$ denotes the parallel transport (with respect to the connection $\nabla^{\mc{E}}$) along the flowline $(\varphi_s(x))_{s \in [0,t]}$. If $f \in C^\infty(\M,\mc{E})$, then $(e^{-t \X} f) (x) = C(\varphi_{-t}(x),t)(f(\varphi_{-t}(x))$. If $\X = X$ is simply the vector field acting on functions (i.e. $\mc{E}$ is the trivial line bundle), then $e^{-tX}f = f(\varphi_{-t}(\cdot))$ is nothing but the composition with the flow.

The resolvents can be meromorphically extended to the whole complex plane by making $\X$ act one anisotropic Sobolev spaces $\mc{H}^s_\pm$. The poles of the resolvents are called the \emph{Pollicott-Ruelle resonances} and have been widely studied in the aforementioned literature \cite{Liverani-04, Gouezel-Liverani-06,Butterley-Liverani-07,Faure-Roy-Sjostrand-08,Faure-Sjostrand-11,Faure-Tsuji-13,Dyatlov-Zworski-16}. Note that the resonances and the resonant states associated to them are intrinsic to the flow and do not depend on any choice of construction of the anisotropic Sobolev spaces. More precisely, there exists a constant $c > 0$ such that $\RR_\pm(z) \in \mc{L}(\mc{H}_\pm^s)$ are meromorphic in $\left\{\Re(z) > -cs\right\}$. For $\RR_+(z)$ (resp. $\RR_-(z)$), the space $\mc{H}^s_+$ (resp. $\mc{H}^s_-$) consists of distributions which are microlocally $H^s$ in a neighborhood of $E_s^*$ (resp. $H^{-s}$ in a neighborhood of $E_s^*$) and microlocally $H^{-s}$ in a neighborhood of $E_u^*$ (resp. $H^{s}$ in a neighborhood of $E_u^*$), see \cite{Faure-Sjostrand-11,Dyatlov-Zworski-16}. These spaces also satisfy $(\mc{H}^s_+)' = \mc{H}^s_-$ (where one identifies the spaces using the $L^2$-pairing). These resolvents satisfy the following equalities on $\mc{H}^s_\pm$, for $z$ not a resonance:
\begin{equation}
\label{equation:resolvent-identity}
\RR_\pm(z)(\mp\X- z) = (\mp\X- z)^{-1} \RR_\pm(z) = \mathbbm{1}_{\mc{E}}
\end{equation}
Given $z \in \C$, not a resonance, we have:
\[
\RR_+(z)^* = \RR_-(\overline{z}),
\]
where this is understood in the following way: given $f_1, f_2 \in C^\infty(\M,\mc{E})$, we have
\[
\langle \RR_+(z) f_1, f_2 \rangle_{L^2} = \langle f_1,  \RR_-(\overline{z}) f_2 \rangle_{L^2}.
\]
(We will always use this convention for the definition of the adjoint.) Since the operators are skew-adjoint on $L^2$, all the resonances (for both the positive and the negative resolvents $\RR_\pm$) are contained in $\left\{\Re(z) \leq 0 \right\}$. A point $z_0 \in \C$ is a resonance for $-\X$ (resp. $\X$) i.e. is a pole of $z \mapsto \RR_+(z)$ (resp. $\RR_-(z)$) if and only if there exists a non-zero $u \in \mc{H}^s_+$ (resp. $\mc{H}^s_-$) for some $s > 0$ such that $-\X u = z_0 u$ (resp. $\X u = z_0 u$). If $\gamma$ is a small counter clock-wise oriented circle around $z_0$, then the spectral projector onto the resonant states is
\[
\Pi_{z_0}^{\pm} = - \dfrac{1}{2\pi i} \int_{\gamma} \RR_{\pm}(z) \dd z =  \dfrac{1}{2\pi i} \int_{\gamma} (z \pm \X)^{-1} \dd z,
\]
where we use the abuse of notation that $-(\X+z)^{-1}$ (resp. $(\X-z)^{-1}$) to denote the meromorphic extension of $\RR_+(z)$ (resp. $\RR_-(z)$).

\subsection{Resonances at $z=0$}

By the previous paragraph, we can write in a neighborhood of $z=0$ the following Laurent expansion (beware the sign conventions):
\[
\RR_+(z) = - \RR_0^+ - \dfrac{\Pi_0^+}{z} + \mc{O}(z).
\]
(Or in other words, using our abuse of notations, $(\X+z)^{-1} = \RR_0^+ + \Pi_0^+/z + \mc{O}(z)$.) And:
\[
\RR_-(z) = -\RR_0^- - \dfrac{\Pi_0^-}{z} + \mc{O}(z).
\]
(Or in other words, using our abuse of notations, $(z - \X)^{-1} = \RR_0^- + \Pi_0^-/z + \mc{O}(z)$.) As a consequence, these equalities define the two operators $\RR_0^{\pm}$ as the holomorphic part (at $z=0$) of the resolvents $-\RR_\pm(z)$. We introduce:
\begin{equation}
\label{equation:pi}
\Pi := \RR_0^+ + \RR_0^-.
\end{equation}
We have the:

\begin{lemma}
\label{lemma:relations-resolvent}
We have $(\RR_0^+)^* = \RR_0^{-}, (\Pi_0^+)^* = \Pi_0^- = \Pi_0^+$. Thus $\Pi$ is formally self-adjoint. Moreover, it is nonnegative in the sense that for all $f \in C^\infty(\M, \mc{E})$, $\langle \Pi f, f \rangle_{L^2} = \langle f, \Pi f \rangle_{L^2} \geq 0$. Eventually, the following statements are equivalent: $\langle{\Pi f, f}\rangle_{L^2} = 0$ if and only if $\Pi f = 0$ if and only if $f = \mathbf{X}u + v$ for some $u \in C^\infty(\M, \E)$ and $v\in \ker(\mathbf{X})$.
\end{lemma}

\begin{proof}
First of all, for $z$ near $0$:
\[
\begin{split}
\RR_+(z)^*  = \RR_-(\overline{z}) & = -\RR_0^- - \Pi_0^-/\overline{z} + \mc{O}(\overline{z}) \\
& = -(\RR_0^+)^* - (\Pi_0^+)^*/\overline{z} + \mc{O}(\overline{z}),
\end{split}
\]
which proves $(\RR_0^+)^* = \RR_0^{-}, (\Pi_0^+)^* = \Pi_0^-$. 

We now show that $\Pi_0^+ = \Pi_0^-$. Since $\X$ is skew-adjoint, we know by \cite[Lemma 2.3]{Dyatlov-Zworski-17} that resonant states at $0$ are smooth. Therefore, for any $s >0$
\[
\ker (-\X|_{\mc{H}^s_+}) = \ker (\X|_{\mc{H}^s_-}) = \ran(\Pi_0^-|_{C^\infty(\M,\E)}) = \ran(\Pi_0^+|_{C^\infty(\M,\E)})
\]
(since $C^\infty(\M,\E)$ is dense in anisotropic Sobolev spaces). Moreover, $\ker(\Pi_0^-|_{C^\infty(\M,\E)}) = \ker(\Pi_0^+|_{C^\infty(\M,\E)})$. Indeed, if $f_1 \in C^\infty(\M,\E) \cap \ker(\Pi_0^-)$, then for any $f_2 \in C^\infty(\M,\E)$, one has $0 = \langle \Pi_0^- f_1, f_2 \rangle_{L^2} = \langle f_1, \Pi_0^+ f_2 \rangle_{L^2}$, that is $f_1$ is orthogonal to $\ran(\Pi_0^+) = \ran(\Pi_0^-)$ and thus for any $f_2$, $0 = \langle f_1, \Pi_0^- f_2 \rangle_{L^2} = \langle \Pi_0^+ f_1, f_2 \rangle_{L^2}$, so $f_1 \in C^\infty(\M,\E) \cap \ker(\Pi_0^+)$. As a consequence, the two projections agree on smooth sections.

To show the nonnegativity, we apply Stone's formula to the self-adjoint operator $i\X$ (with dense domain $\mc{D}_{L^2}$ previously defined in \eqref{equation:domaine-p}). More precisely, taking $\mc{H} := L^2(\M, \mc{E} ; \dd \mu) \cap \ker \Pi_0^+$, the spectrum of $i\X$ on $\mc{H}$ (near the spectral value $0$) is only absolutely continuous and if $\pi_{[a,b]}$ denotes the spectral projection onto the energies $[a,b]$, we obtain:
\[
\begin{split}
\pi_{[a,b]} & = \lim_{\eps \rightarrow 0} \dfrac{1}{2 \pi i} \int_a^b \left((i\X -(\lambda+i \eps))^{-1} - (i\X -(\lambda-i\eps))^{-1}\right) \dd \lambda \\
& = \lim_{\eps \rightarrow 0} \dfrac{1}{2\pi} \int_a^b \left(- \RR_-(-i\lambda + \eps) - \RR_+(i \lambda + \eps) \right) \dd \lambda \\
& = - \dfrac{1}{2\pi} \int_a^b \left( \RR_-(-i \lambda) + \RR_+(i \lambda)\right) \dd \lambda,
\end{split}
\]
where the limit is understood in the weak sense (by applying it to $f \in C^\infty(SM,\mc{E}) \cap \ker \Pi_0^+$ and pairing it to $f$). We then obtain:
\[
\partial_\lambda \pi_{(-\infty,\lambda)}|_{\lambda=0} = \dfrac{1}{2\pi} (\RR_0^- + \RR_0^+) = \dfrac{\Pi}{2\pi}\geq 0.
\]

Assume $\langle{\Pi f, f}\rangle_{L^2} = 0$ for some $f\in C^\infty(\M, \E)$; as $\RR_0^- = (\RR_0^+)^*$, equivalently we have $\Re (\langle{\RR_0^+f, f}\rangle_{L^2}) = 0$. Using the fact that $\mc{H}_+^s = \ker \Pi_0^{+} \oplus \ran \Pi_0^+$ for any $s > 0$, as well as the relation $\mathbf{X} \RR_0^+ = \mathbbm{1} - \Pi_0^+$ given in \eqref{eq:resolventrelations} below, we have that $\mathbf{X}: \ker \Pi_0^+ \xrightarrow{\cong} \ker \Pi_0^+$ is an isomorphism with inverse $\pm \RR_0^+$. Thus setting $u := \pm \RR_0^+ f$ and $v := \Pi_0^+ f$, we may write $f = \X u + v$. We compute
\begin{equation}\label{eq:Stonevsscattering}
	0 = \Re (\langle{\RR_0^+ f, f}\rangle_{L^2}) = \Re(\langle{u, \Pi_0^+f + \X u}\rangle_{L^2}) = \Re(\langle{u, \X u}\rangle_{L^2}) = -\Im(\langle{-i\X u, u}\rangle_{L^2}),
\end{equation}
using that $\Pi_0^+$ is formally self-adjoint and $u \in \ker \Pi_0^+$. Since $f \in C^\infty(\M, \E) \subset \mc{H}_+^s$ for any $s > 0$, we have $u \in \mc{H}_+^s$ for any $s > 0$, and so the wavefront set of $u$ satisfies $\WF(u) \subset E_u^*$. Thus again an application of \cite[Lemma 2.3]{Dyatlov-Zworski-17} gives $u \in C^\infty$. It is then immediate that $\Pi f = 0$, thus completing the proof.\footnote{Note that the positivity of $\Pi$ alternatively follows from \eqref{eq:Stonevsscattering} and Lemma \cite[Lemma 2.3]{Dyatlov-Zworski-17}.}
\end{proof}

In the following, we will write $\ker \X$ instead of $\ker \X|_{\mc{H}^s_{\pm}}$ in order not to burden the notations, but be careful that we are always referring to elements in anisotropic spaces (otherwise, $\ker \X|_{H^{-s}}$ is infinite dimensional for any $s > 0$). We also record here for the sake of clarity the following identities:
\begin{equation}\label{eq:resolventrelations}
\begin{split}
& \Pi_0^{+} \RR_0^+ = \RR_0^+ \Pi_0^{+} = 0,\,\, \Pi_0^{-} \RR_0^- = \RR_0^- \Pi_0^{-} = 0,\\
& \X \Pi_0^\pm = \Pi_0^\pm \X = 0,\,\, \X \RR_0^+ = \RR_0^+ \X = \mathbbm{1} - \Pi_0^+,\,\, -\X \RR_0^- = -\RR_0^- \X = \mathbbm{1} - \Pi_0^-.
\end{split}
\end{equation}
We also have:

\begin{lemma}
\label{lemma:resonances-zero}
We have:
\begin{enumerate}
\item[1.] If $u \in \ker (\X)$, then $u \in C^\infty(\M,\mc{E})$ and $u$ does not vanish unless $u \equiv 0$,
\item[2.] There exists a basis $u_1,...,u_p$ of $\ker (\X)$ such that
\[
\Pi_0^{\pm} = \sum_{i=1}^p \langle \cdot, u_i \rangle_{L^2} u_i.
\]
\item[3.] Let $u_1, ..., u_p$ be a basis of $\ker (\X)$. Then for all $x \in \M$, the vectors $(u_1(x),...,u_p(x))$ are independent as elements of $\mc{E}_x$. We can thus always assume that $(u_1(x),...,u_p(x))$ are orthonormal.
\item[4.] In particular, $\dim(\ker(\X)) \leq \rk(\mc{E})$.
\end{enumerate}
\end{lemma}

\begin{proof} These are simple consequences of the discussion above:

\emph{Proof of 1.} Just use that $X |u|^2 = \langle \X u, u \rangle + \langle u, \X u \rangle = 0$, thus $|u|^2$ is constant. 

\emph{Proof of 2.} This follows from Lemma \ref{lemma:relations-resolvent}. 

\emph{Proof of 3.} This follows from the fact that $\nabla^{\E}$ is unitary: assume $u_1, \dotso, u_p$ is orthonormal at $x_0$. Then $X\langle{u_i, u_j}\rangle_{\E} = \langle \X u_i, u_j \rangle + \langle u_i, \X u_j \rangle = 0$, so by transitivity $\langle{u_i, u_j}\rangle_{\E} \equiv \delta_{ij}$ globally. 

\emph{Proof of 4.} Follows from the previous item.
\end{proof}

\subsection{General remarks on opaque connections}

\label{ssection:remarks-endomorphisms}

As mentioned earlier, $\nabla^{\mc{E}}$ induces a canonical connection $\nabla^{\mathrm{End}(\E)}$ on $\mathrm{End}(\mc{E})$ defined by \eqref{equation:connection-end}. Since $\nabla^{\mc{E}}$ is assumed to be unitary, we have $\nabla^{\End(\E)}$ is unitary and $X$ preserves $\dd \mu$, so $\nabla^{\mathrm{End}(\E)}_X$ is formally skew-adjoint.
%
As a consequence, the Pollicott-Ruelle resonant states of $\nabla^{\mathrm{End}(\E)}_X$ at $0$ are smooth. Moreover, they always contain the section $\mathbbm{1}_{\mc{E}}$. We want to investigate what are the other resonant states at $0$ of $\nabla^{\mathrm{End}(\E)}_X$ and what does their existence imply. Recall that a subbundle $\mc{F}$ of $\E$ is said to be invariant if for all $x \in \M, f \in \mc{F}_x$, $t \in \mathbb{R}$, one has $C(x,t) f \in \mc{F}_{\varphi_t(x)}$. We say that $\mc{F}$ is \emph{irreducible} if any invariant subbundle of $\mc{F}$ is either $\mc{F}$ or $\left\{ 0 \right\}$.

We have the following straightforward observations:

\begin{lemma}
\label{lemma:observations}
The following hold for a subbundle $\F \subset \E$:
\begin{enumerate}
\item[1.] If $\mc{F}$ is invariant, then so is $\mc{F}^\bot$ (defined pointwise by taking the orthogonal subspace).
\item[2.] $\mc{F}$ is invariant if and only if for all $f \in C^\infty(\M,\mc{F})$, $\nabla^{\mc{E}}_X f \in C^\infty(\M,\mc{F})$.
\item[3.] If $\mc{F}$ is invariant, then one has the splitting: $(\mc{E},\nabla_X^{\mc{E}}) = (\mc{F},\nabla_X^{\mc{E}}|_{\mc{F}}) \oplus (\mc{F}^\bot,\nabla_X^{\mc{E}}|_{\mc{F}^\bot})$.
\item[4.] $\mc{F}$ is invariant if and only if $\nabla^{\mathrm{End}(\E)}_X \Pi_{\mc{F}} = 0$, where $\Pi_{\mc{F}}$ denotes the pointwise orthogonal projection onto $\mc{F}$.
\end{enumerate}
\end{lemma}

\begin{proof} We check the validity of each point separately:

\emph{Proof of 1.} Assume $\mc{F}$ is invariant and consider $x \in \M$ and $f_2 \in \mc{F}^\bot_x$. For $t \in \R$, consider $f_1' \in \mc{F}_{\varphi_t(x)}$; since $\mc{F}$ is invariant, it can be written as $f_1' = C(x,t)f_1$, for some $f_1 \in \mc{F}_x$. Then:
\[
\langle f'_1, C(x,t) f_2 \rangle_{\F_{\varphi_t x}} = \langle C(x,t) f_1, C(x,t)f_2 \rangle_{\F_{\varphi_t x}} = \langle f_1,f_2 \rangle_{\F_x} = 0,
\]
and thus $C(x,t)f_2 \in \mc{F}^\bot_{\varphi_t(x)}$, that is $\mc{F}^\bot$ is invariant.

\emph{Proof of 2.} Assume firstly $\F$ is invariant. Consider $f_1 \in C^\infty(\M,\mc{F})$ and $x\in \M$. Consider $f_2 \in \mc{F}^\bot_x$ and extend $f_2$ by parallel transport along $(\varphi_t(x))_{t \in (-\eps,\eps)}$ for some $\eps > 0$. By the first item, $f_2$ is a section of $\mc{F}^\bot$. Thus:
\[
X \cdot \langle f_1, f_2 \rangle_{\E} (x) = 0 = \langle \nabla^{\mc{E}}_X f_1, f_2 \rangle_{\E_x} + \langle f_1, \underbrace{\nabla^{\mc{E}}_X f_2}_{=0} \rangle_{\E_x},
\]
and thus in particular $\langle \nabla^{\mc{E}}_X f_1, f_2 \rangle_{\E_x} = 0$. Conversely, assume $\mc{F}$ is a subbundle of $\mc{E}$ such that for all $f_1 \in C^\infty(\M,\mc{F})$, $\nabla_X f_1 \in C^\infty(\M,\mc{F})$. This is also true for $\mc{F}^\bot$: indeed, if $f_2 \in C^\infty(\M,\mc{F}^\bot)$, then $\langle f_1, \nabla_X f_2 \rangle_{\E} = X \cdot \langle f_1, f_2 \rangle_{\E} - \langle \nabla_X f_1, f_2 \rangle_{\E} = 0$, i.e. $\nabla_X f_2 \in C^\infty(\M,\mc{F}^\bot)$. Now, consider $x_0 \in \M$, a local chart $\Omega_{x_0}$ around $x_0$, and a local orthonormal frame $(e_1,...,e_r)$ of $\E|_{\Omega_{x_0}} = \Omega_{x_0} \times \mathbb{C}^r$
 such that $(e_1,...,e_k)$ is a frame for $\mc{F}|_{\Omega_{x_0}}$ and $(e_{k+1},...,e_r)$ a frame for $\mc{F}^\bot|_{\Omega_{x_0}}$. 
On $\Omega_{x_0}$, the connection can be written as $\nabla^{\mc{E}} = d + \Gamma$. We claim that for every $x \in \Omega_{x_0}$, 
$\Gamma(X)(\F_x) \subset \F_x$. 
Indeed, 
consider $f = \sum_{i=1}^k f_i e_i \in \mc{F}_x$ and smooth functions $\widetilde{f}_1,...,\widetilde{f}_k$ defined around $x_0$ such that $\dd \widetilde{f}_i(x) = 0$ and $\widetilde{f}_i(x) = f_i$, and set $\widetilde{f} := \sum_{i=1}^k \widetilde{f}_i e_i$. Then $\nabla^{\mc{E}}_X \widetilde{f} (x) = \Gamma(X) \widetilde{f} (x) = \Gamma(X) f \in \mc{F}_x$ by assumption. Analogously, 
we have $\Gamma(X)(\F_x^\bot) \subset \F_x^\bot$ for every $x \in \Omega_{x_0}$. We then obtain that for $f \in \mc{E}_x$ and $x \in  \Omega_{x_0}$, writing $f(t) := C(x,t)f = (f_1(t),f_2(t))$ with $(f_1(t),0) \in \mc{F}_{\varphi_t(x)}, (0,f_2(t)) \in \mc{F}^\bot_{\varphi_t(x)}$, we have two separate differential equations for the parallel transport: $\dot{f}_1(t) = -\Gamma_{\mc{F}}(t)f_1(t), \dot{f}_2(t) = -\Gamma_{\mc{F}^\bot}(t) f_2(t)$. As a consequence, if $f_2(0)=0$, then $f_2(t) = 0$ for all $t$, which proves the claim.

\vspace{1pt}

\emph{Proof of 3.} This is immediate by the previous item.
\vspace{1pt}

\emph{Proof of 4.} Assume $\mc{F}$ is invariant. Then any $f \in C^\infty(M,\mc{E})$ can be decomposed as $f=f_1+f_2$, where $f_1 = \Pi_{\mc{F}}f, f_2 = \Pi_{\mc{F}^\bot}f$ and by the second item:
\[
(\nabla^{\End (\E)}_X \Pi_{\mc{F}})f = \nabla^{\mc{E}}_X (\Pi_{\mc{F}}f) - \Pi_{\mc{F}}(\nabla^{\mc{E}}_X f) =  \nabla^{\mc{E}}_X f_1 -  \Pi_{\mc{F}}(\underbrace{\nabla^{\mc{E}}_X f_1}_{\in \mc{F}} + \underbrace{\nabla^{\mc{E}}_X f_2}_{\in \mc{F}^\bot}) = \nabla^{\mc{E}}_X f_1 -\nabla^{\mc{E}}_X f_1 = 0.
\]
Conversely, if $\nabla^{\End(\E)}_X \Pi_{\mc{F}} = 0$, then for any $f \in C^\infty(\M,\mc{F})$, one has
\[
0 = (\nabla^{\End(\E)}_X \Pi_{\mc{F}}) f = \nabla^{\mc{E}}_X f -  \Pi_{\mc{F}}(\nabla^{\mc{E}}_X f),
\]
that is $\nabla^{\mc{E}}_X f \in C^\infty(\M,\mc{F})$ so $\mc{F}$ is invariant by the second item.
\end{proof}

\begin{remark}
As mentioned earlier, in the usual terminology, a connection such that the holonomy is trivial on all closed orbits is called \emph{transparent} (see \cite{Paternain-09,Paternain-11,Paternain-12, Paternain-13} for further details). This implies the existence of a global smooth orthonormal basis $(e_1,...,e_r)$ of $\mc{E}$ such that $\nabla^{\mc{E}}_X e_i = 0$ (see \cite{Cekic-Lefeuvre-20}), that is $\ker \nabla^{\mc{E}}_X$ has maximal dimension equal to $r$ (by the last item of Lemma \ref{lemma:resonances-zero}). These elements give rise to a global basis $(e_i^* \otimes e_j)_{i,j=1}^r$ of $\mathrm{End}(\E)$ such that $\nabla^{\mathrm{End}(\E)}_X (e_i^* \otimes e_j)=0$, that is $\ker \nabla^{\mathrm{End}(\E)}_X$ has also maximal dimension equal to $r^2$. 
\end{remark}

\begin{example}
Here are some examples to illustrate the terminology: if $(M,g)$ is a surface, and $\pi : \M = SM \rightarrow M$ denotes the projection, then the geodesic lift of the vector bundle $\pi^* TM \rightarrow SM$ is transparent. If $(M,g)$ is an $n$-manifold for $n \geq 3$, then the geodesic lift of $TM \rightarrow M$ might not be transparent and is never opaque. Indeed, the direction of the geodesic vector field is always preserved (i.e. the section $s(x, v) = v$ satisfies $(\pi^*\nabla_{\mathrm{LC}})_X s = 0$, where $\nabla_{\mathrm{LC}}$ is the Levi-Civita connection) but there may be some holonomy transversally to it. Beware, if $\mc{E} \to \M$ is a line bundle, it is always opaque but might not be transparent.
\end{example}
%

If $u \in C^\infty(\M,\mathrm{End}(\E))$, then $u = u_R + u_I$, where $u_R := \frac{u + u^*}{2}$ is hermitian and $u_I := \frac{u-u^*}{2}$ is skew-Hermitian. Since $\nabla^{\mathrm{End}(\E)}_X(u^*) = (\nabla^{\mathrm{End}(\E)}_X u)^*$ (see Lemma \ref{lemma:commutation-adjoint}), one obtains that $\nabla^{\mathrm{End}(\E)}_X u = \nabla^{\mathrm{End}(\E)}_X u_R + \nabla^{\mathrm{End}(\E)}_X u_I$ is the decomposition into Hermitian and skew-Hermitian parts of $\nabla^{\mathrm{End}(\E)}_X u$. Thus, $\nabla^{\mathrm{End}(\E)}_X u = 0$, if and only if $u = u_1 + i u_2$, where $\nabla^{\mathrm{End}(\E)}_X u_j = 0$ and $u_j^*=u_j$ for $j = 1, 2$. In other words,
\begin{equation}\label{eq:skewHermitiandiscussion}
\ker(\nabla^{\mathrm{End}(\E)}_X)_\C =\left(\ker(\nabla^{\mathrm{End}(\E)}_X) \cap \ker({\bullet}^*-\mathbbm{1}_{\mc{E}})\right)_\R \oplus i \times \left(\ker(\nabla^{\mathrm{End}(\E)}_X) \cap \ker({\bullet}^*-\mathbbm{1}_{\mc{E}})\right)_\R,
\end{equation}
where the subscript $\R$ or $\C$ indicates that it is seen as an $\R$- or $\C$-vector space. We have the following picture:

\begin{lemma}
\label{lemma:seldjointdecomposition}
If $u \in \ker(\nabla^{\mathrm{End}(\E)}_X),u=u^*$, then:
\begin{itemize}
\item At each point $x \in \M$, there exists a smooth orthogonal splitting $\mc{E}_x = {\oplus_{i=1}^k} \mc{E}_i(x)$ such that each $\mc{E}_i$ is 
invariant and $\mc{E}_i \rightarrow \M$ is a well-defined subbundle of $\mc{E} \rightarrow \M$,
\item For all $x \in \M$, $u(x) = \sum_{i=1}^k \lambda_i \Pi_i(x)$, where $\Pi_i(x)$ is the orthogonal projection onto $\mc{E}_i$ (with kernel $\oplus_{j=1,j\neq i}^k \mc{E}_i)$, $\lambda_i$ are the \emph{distinct} eigenvalues of $u$,
\item Each projection satisfies $\nabla^{\mathrm{End}(\E)}_X \Pi_i = 0$.
\end{itemize}
\end{lemma}

Note that a pedantic way of reformulating the fact that $\nabla^{\mc{E}}$ is transparent would be to say that there exists a (maximally) invariant orthogonal splitting such that for all $x \in \M$:
\[
\mc{E}_x = {\oplus_{i=1}^r}^\bot \mc{E}_i(x).
\]
Moreover, $\mc{E}_i = \C \cdot e_i$ for some $e_i \in C^\infty(SM,\mc{E})$ such that $\nabla_X e_i = 0$.

\begin{proof}
Consider a dense orbit $\mc{O}(x_0)$, and a basis $(e_i)_{i=1}^r$ of $\mc{E}|_{\mc{O}(x_0)}$ that is invariant by parallel transport along the orbit. Then $u$ can be written as $u = \sum_{i,j=1}^r \lambda_{ij} e_i^* \otimes e_j$ for some smooth functions $\lambda_{ij} \in C^\infty(\mc{O}(x_0))$ and:
\[
\begin{split}
\nabla^{\End(\E)}_X u &  = \sum_{i,j=1}^r X \lambda_{ij} e_i^* \otimes e_j + \sum_{i,j=1}^r  \lambda_{ij} (\nabla_X e_i)^* \otimes e_j + \sum_{i,j=1}^r  \lambda_{ij} e_i^* \otimes  \nabla_X e_j \\
& = \sum_{i,j=1}^r X \lambda_{ij} e_i^* \otimes e_j = 0,
\end{split}
\]
thus $\lambda_{ij}$ are constant along $\mc{O}(x_0)$. This implies that the distinct eigenvalues of $u$ are constant along $\mc{O}(x_0)$ and thus constant on $\M$ (the eigenvalues counted with multiplicity are continuous on $\M$, thus uniformly continuous since $\M$ is compact; since they are constant on a dense set, they are constant everywhere). We denote the distinct ones by $\lambda_1, ..., \lambda_k$ and introduce for all $x \in \M$:
\[
\Pi_i(x) := \dfrac{1}{2 \pi i} \int_{\gamma_i} (u(x)-\lambda_i \mathbbm{1}_{\mc{E}})^{-1} \dd \lambda,
\]
where $\gamma_i$ is a small (counter clockwise oriented) circle around $\lambda_i$. One has: $u = \sum_i \lambda_i \Pi_i$. Observe that
\[
\nabla^{\End (\E)}_X \Pi_i =- \dfrac{1}{2 \pi i} \int_{\gamma_i} (u(x)-\lambda_i \mathbbm{1}_{\mc{E}})^{-1} \left(\nabla^{\End(\E)}_X (u(x)-\lambda_i \mathbbm{1}_{\mc{E}})\right) (u(x)-\lambda_i \mathbbm{1}_{\mc{E}})^{-1} \dd \lambda = 0.
\]
\end{proof}

We have the following characterization of opaque connections:

\begin{lemma}
\label{lemma:characterization-opaque}
The connection $\nabla^{\mc{E}}$ is opaque if and only if the Pollicott-Ruelle resonant states of $\nabla^{\End(\E)}_X$ are reduced to $\mathbbm{1}_{\mc{E}}$ i.e. $\ker (\nabla^{\End(\E)}_X|_{\mc{H}^s_\pm}) = \C\cdot \mathbbm{1}_{\mc{E}}$.
\end{lemma}

\begin{proof}
``$\implies$'' Assume that the connection is opaque and $\ker (\nabla^{\End(\E)}_X|_{\mc{H}^s_\pm}) \neq \C\cdot \mathbbm{1}_{\mc{E}}$, then one can consider $0 \neq u \in \ker (\nabla^{\End(\E)}_X|_{\mc{H}^s_\pm})$ which is orthogonal to $\C\cdot \mathbbm{1}_{\mc{E}}$ (i.e. its trace vanishes everywhere on $\M$). 
Taking its self-adjoint or $i$ times the skew-adjoint part, by the previous discussion we may additionally assume $u^* = u$ and $u \neq 0$. By Lemma \ref{lemma:seldjointdecomposition}, it can be decomposed as $u = \sum_{i=1}^k \lambda_i \Pi_i$, where each $\Pi_i$ is the orthogonal projection corresponding to an invariant subbundle $\mc{E}_i \rightarrow \M$, i.e. $\nabla^{\End(\E)}_X \Pi_i = 0$. Observe that since $\Tr(u) = 0$, this decomposition cannot be the trivial one i.e. $\mc{E} = \mc{E} \oplus^\bot \left\{ 0 \right\}$ (in which case $u$ would be a multiple of $\mathbbm{1}_{\mc{E}}$). Thus, $\mc{E}_1$ is an invariant subbundle which is neither $\left\{0 \right\}$ nor $\mc{E}$ which contradicts the fact that the connection is opaque. \\

`` $\Longleftarrow$'' Conversely, if the connection is not opaque, then it admits an invariant subbundle $\mc{F}$ and $\mc{E} = \mc{F} \oplus^\bot \mc{F}^\bot$ is an invariant decomposition. The orthogonal projection $\Pi_{\mc{F}}$ satisfies $\nabla^{\End(\E)}_X \Pi_{\mc{F}} = 0$ by Lemma \ref{lemma:observations}, thus $\ker (\nabla^{\End(\E)}_X|_{\mc{H}^s_\pm}) \neq \C\cdot \mathbbm{1}_{\mc{E}}$.
\end{proof}

\subsection{Proof of the genericity of opaque connections in the geodesic case}


Here, we consider the geodesic case $\M = SM$ and $\mc{E} \rightarrow M$ is a Hermitian vector bundle of rank $r \geq 1$ with unitary connection $\nabla^{\mc{E}}$. As in the previous sections, we consider the pullback bundle $\pi^* \E \rightarrow SM$ (where $\pi : SM \rightarrow M$ is the projection) equipped with the pullback connection $\pi^* \nabla^{\mc{E}}$. The parallel transport is considered with respect to this connection along the flowlines of the geodesic vector field $X$. In this paragraph, we prove Theorem \ref{theorem:generic-opacity}, namely that there is an open dense subset of $C^{k}$ unitary connections (for $k \geq 2$) on $\E \rightarrow M$ such that the induced connection on $SM$ is opaque.

\begin{proof}[Proof of Theorem \ref{theorem:generic-opacity}]
Openness can be seen from Lemma \ref{lemma:characterization-opaque}. Indeed, if $\nabla^{\mc{E}}$ is opaque then $\X := (\pi^*\nabla^{\mathrm{End}(\E)})_X$ has $\mathbbm{1}_{\mc{E}}$ as only resonant state at $0$. By continuity of the resonances (see \cite{Bonthonneau-19}), this is still the case for any $\X_A := (\pi^*(\nabla^{\mathrm{End}(\E)} + [A,\cdot]))_X$ induced by a perturbation $\nabla^{\mc{E}} +A$.\footnote{The argument is much easier compared to \cite{Bonthonneau-19}: we can keep the anisotropic Sobolev space $\mc{H}_+^s$ fixed and consider $[A(X),\cdot]$ as a perturbation given by a potential.} Density follows from a perturbation argument of spectral theory: we use Lemma \ref{lemma:characterization-opaque} and show that a well-chosen perturbation ejects all the resonances at $0$ of $\nabla^{\End(\E)}_X$, except the one given by identity $\mathbbm{1}_{\E}$.

We now assume that $\nabla^{\mc{E}}$ is not opaque. We define $\nabla_s := \nabla^{\mc{E}} + s A$, where $A \in C^\infty(M,T^*M \otimes \mathrm{End}_{\mathrm{sk}}(\mc{E}))$. We set $\nabla^{\End(\E)}_s := \nabla^{\mathrm{End}(\E)} + s [A,\cdot]$ and thus $(\pi^* \nabla^{\mathrm{End}(\E)}_s)_X = (\pi^*\nabla^{\mathrm{End}(\E)})_X + s [\pi^*A(X),\cdot]$. We introduce $\X_s := (\pi^* \nabla^{\mathrm{End}(\E)}_s)_X$, $\X := \X_0$. Following \S\ref{ssection:resonances}, the operator $\C \ni z \mapsto (-\X_s -z)^{-1}=: \RR_{+}(z,s)$ is meromorphic with poles in $\left\{\Re(z) \leq 0\right\}$ of finite rank. We introduce:
\begin{equation}\label{eq:Pi_0^+def}
\Pi^+_s := - \dfrac{1}{2\pi i} \int_\gamma \RR_{+}(z,s) \dd z = \dfrac{1}{2\pi i} \int_\gamma (\X_s + z)^{-1} \dd z,
\end{equation}
where $\gamma$ is a small counter-clockwise oriented circle centred around $0$ and we use the abuse of notation $-(\X_s + z)^{-1} = \RR_{+}(z,s)$ for the meromorphic extension from $\left\{\Re(z) > 0\right\}$ to $\C$ using the anisotropic spaces $\mc{H}^s_+$. For $s=0$, we have $\Pi^+_{s=0} = \Pi^+_0$ the $L^2$-orthogonal projection on the resonant states at $0$ but for $s\neq 0$, this is the projection onto the (direct) sum of all the resonant states generated by resonances inside the small circle $\gamma$. We set:
\[
\lambda^+_s := \Tr(-\X_s \Pi^+_s) = \Tr\left(-\X_s  \dfrac{1}{2\pi i} \int_\gamma (\X_s + z)^{-1} \dd z \right)
\]
This is the sum of the resonances inside the small circle $\gamma$ (note that the resonances are necessarily contained in $\left\{\Re(z) \leq 0\right\}$). The map $\R \ni s \mapsto \lambda^+_s \in \C$ is (at least) $C^2$ if $A$ is smooth. We introduce the notation $P_A : C^\infty(SM,\mathrm{End}(\mc{E})) \rightarrow C^\infty(SM, \mathrm{End}(\E))$ for the map
\begin{equation}
\label{equation:commutator-p}
P_Au(x,v) := [A_x(v),u(x,v)],
\end{equation}
so that $\X_s = \X + sP_A$. We now compute the first- and second-order derivatives of $s \mapsto \lambda^+_s$ at $s=0$:

\begin{lemma}
\label{lemma:perturbation-dx-1}
We have:
\[
\dot{\lambda}^+_0 = -\Tr(P_A \Pi^+_0) = -\sum_{i=1}^p \langle P_A u_i, u_i \rangle_{L^2} = 0,
\]
where $u_1, ..., u_p$ is a smooth orthonormal basis of $\ker(\X)$.
\end{lemma}

\begin{proof}
By the definition \eqref{eq:Pi_0^+def} of $\Pi_0^+$, we compute $\dot{\Pi}_0^+ = - \RR_0^+ P_A \Pi_0^+ - \Pi_0^+ P_A \RR_0^+$. Using the relation $\dot{\lambda}_0^+ = -\Tr(P_A \Pi_0^+ + \X \dot{\Pi}_0^+)$, the cyclicality of the trace and \eqref{eq:resolventrelations}, the first two equalities follow. As to the fact that it is $0$, it follows from the following argument showing that each term is the sum is actually $0$. First of all, we can always assume that the $u_i$ are either \emph{odd} or \emph{even} i.e. they only have odd or even terms in their Fourier expansions. Indeed, splitting between odd/even, the equation $\X u = 0$ yields $\X u_{\mathrm{even}} = 0, \X u_{\mathrm{odd}} = 0$ and odd/even functions are orthogonal with respect to the $L^2$-scalar product. As a consequence, we can always assume that the $u_i$ are either odd or even. 
Assume $u_i(x, -v) = (-1)^k u_i(x, v)$ for some $k \in \{0, 1\}$. Then, writing $\dd \mu$ for the Liouville measure on $SM$ and making the change of variable $v' = -v$:
\[
\begin{split}
\langle P_A u_i, u_i \rangle_{L^2(SM)} & = \int_{SM} \Tr((P_A u_i) u_i^*)(x,v) \dd \mu(x,v) \\
& = \int_{SM} \Tr([A_x(v), u_i(x,v)] u_i^*(x,v)) \dd \mu(x,v) \\
& =- \int_{SM} \Tr([A_x(v), u_i(x,v)] u_i^*(x,v)) \dd \mu(x,v) = -\langle P_A u_i, u_i \rangle_{L^2},
\end{split}
\]
Thus $\langle P_A u_i, u_i \rangle_{L^2} = 0$, completing the proof. 
\end{proof}

As far as the second variation is concerned, we have:

\begin{lemma}
\label{lemma:perturbation-dx-2}
We have:
\[
\ddot{\lambda}^+_0 = 2 \Tr(\Pi^+_0 P_A \RR_0^+ P_A \Pi^+_0) = - 2 \sum_{i=1}^p \langle \RR_0^+ P_A u_i, P_A u_i \rangle_{L^2}.
\]
\end{lemma}

\begin{proof}
We start with:
\[
\ddot{\lambda}^+_0  = -\Tr(\ddot{\X}_0 \Pi_0^+) -2 \Tr(\dot{\X}_0 \dot{\Pi}^+_0) - \Tr(\X_0 \ddot{\Pi}^+_0).
\]
It is immediate that $\ddot{\X}_0 = 0$ since the variation is linear. We then compute:
\[
\dot{\Pi}^+_0 = - \Pi^+_0 \dot{\X}_0 \RR_0^+ - \RR_0^+ \dot{\X}_0 \Pi^+_0.
\]
Thus, by the cyclicity of trace:
\[
-2 \Tr(\dot{\X}_0 \dot{\Pi}^+_0) = 2 \Tr( \dot{\X}_0 \Pi^+_0 \dot{\X}_0 \RR_0^+ + \dot{\X}_0  \RR_0^+ \dot{\X}_0 \Pi^+_0) = 4 \Tr( \Pi^+_0  \dot{\X}_0  \RR_0^+ \dot{\X}_0 \Pi^+_0).
\]
And last but not least:
\begin{align*}
\ddot{\Pi}^+_0 &= 2 (\RR_0^+ \dot{\X}_0 \RR_0^+ \dot{\X}_0  \Pi_0^+ + \RR_0^+ \dot{\X}_0 \Pi_0^+ \dot{\X}_0 \RR_0^+ +  \Pi_0^+ \dot{\X}_0 \RR_0^+ \dot{\X}_0 \RR_0^+)\\
&+ 2(\Pi_0^+ \dot{\X}_0 \Pi_0^+ \dot{\X}_0 H_1 + \Pi_0^+ \dot{\X}_0 H_1 \dot{\X}_0 \Pi_0^+ + H_1 \dot{\X}_0 \Pi_0^+ \dot{\X}_0 \Pi_0^+).
\end{align*}
Here $H_1$ denotes the coefficient next to $z$ in the holomorphic expansion of $(\X + z)^{-1}$ close to zero, i.e. $(\X + z)^{-1} = \frac{\Pi_0^+}{z} + \RR_0^+ + zH_1 + \mc{O}(z^2)$. Thus, using the cyclicity of the trace, the fact that $\X \Pi^+_0 = \Pi^+_0 \X = 0$, $\RR_0^+ \Pi^+_0 = \Pi^+_0 \RR_0^+ = 0$ and $\X \RR_0^+ = \mathbbm{1} - \Pi_0^+$ (see \eqref{eq:resolventrelations}) we obtain:
\[
\begin{split}
\Tr(\X \ddot{\Pi}^+_0) & = 2 \times \Tr\Big(\X (\RR_0^+ \dot{\X}_0 \RR_0^+ \dot{\X}_0  \Pi_0^+ + \Pi_0^+ \dot{\X}_0 \RR_0^+ \dot{\X}_0 \RR_0^+ + \RR_0^+ \dot{\X}_0 \Pi_0^+ \dot{\X}_0 \RR_0^+)\Big) \\
& = 2 \Tr(\X \RR_0^+  \dot{\X}_0 \Pi_0^+ \dot{X}_0 \RR_0^+ ) \\
& = 2 \Tr((1-\Pi_0^+)  \dot{\X}_0 \Pi_0^+ \dot{\X}_0 \RR_0^+) \\
& = 2 \Tr(\dot{\X}_0 \Pi_0^+ \dot{\X} \RR_0^+) - 2 \Tr(\Pi_0^+ \dot{\X}_0 \Pi_0^+ \dot{\X}_0 \RR_0^+) \\
& = 2 \Tr(\Pi_0^+ \dot{\X}_0 \RR_0^+ \dot{\X}_0  \Pi_0^+).
\end{split}
\]
Thus, taking $u_1, ...,u_p$ an orthonormal basis of $\ker(\X)$ and using that $\dot{\X}_0^* = -\dot{\X}_0$ and $P_A = \dot{\X}_0$, we obtain:
\[
\begin{split}
\ddot{\lambda}^+_0 &= 2 \Tr(\Pi_0^+ \dot{\X}_0 \RR_0^+ \dot{\X}_0  \Pi_0^+) \\
 &= 2 \sum_{i=1}^p \langle \dot{\X}_0 \RR_0^+ \dot{\X}_0 u_i, u_i \rangle_{L^2} = - 2 \sum_{i=1}^p \langle  \RR_0^+ P_A u_i,  P_A u_i \rangle_{L^2}.
\end{split}
\]
\end{proof}
\noindent We could have introduced the other sum of resonances
\[
\lambda_s^- := \Tr(\X_s \Pi_s^-),
\]
using the other resolvent
\[
\Pi_s^- :=  - \dfrac{1}{2\pi i} \int_\gamma \RR_-(z,s) \dd z = \dfrac{1}{2\pi i} \int_\gamma (z - \X_s)^{-1} \dd z,
\]
still using the same abuse of notation. We have:

\begin{lemma}
$\overline{\lambda_s^-} = \lambda_s^+$.
\end{lemma}

\begin{proof}
By Lemma \ref{lemma:relations-resolvent} we have $(\Pi_s^+)^*  = \Pi_s^-$. Thus:
\[
\overline{\lambda_s^+} = \overline{ \Tr(- \X_s \Pi_s^+)} = \Tr((-\X_s \Pi_s^+)^*) = \Tr(\X_s (\Pi_s^+)^*) = \Tr(\X_s \Pi^-_s) = \lambda_s^{-}.
\]
\end{proof}

A similar computation to the one carried out in Lemmas \ref{lemma:perturbation-dx-1} and \ref{lemma:perturbation-dx-2}, also gives that $\dot{\lambda}^-_0 = 0$ and:
\[
\ddot{\lambda}^-_0 = - 2 \sum_{i=1}^p \langle \RR_0^- P_A u_i, P_A u_i \rangle_{L^2}
\]
As a consequence, we obtain the formula, writing $\lambda_s := \lambda_s^+$
\[
\Re(\ddot{\lambda}_0) = - \sum_{i=1}^p \langle \Pi P_A u_i, P_A u_i \rangle_{L^2},
\]
where $\Pi$ was introduced in \eqref{equation:pi}. This can be rewritten, combining Lemmas \ref{lemma:perturbation-dx-1} and \ref{lemma:perturbation-dx-2}:
\[
\Re(\lambda_s) = \Re(\lambda_0) + s \Re(\dot{\lambda}_0) + 1/2 \times s^2 \Re(\ddot{\lambda}_0) + \mc{O}(s^3)  =  - 1/2 \times s^2 \sum_{i=1}^p \langle \Pi P_A u_i, P_A u_i \rangle_{L^2} + \mc{O}(s^3),
\]
where we recall that $P_A$ is defined in \eqref{equation:commutator-p}. Since $\Pi$ is non-negative and formally self-adjoint (see Lemma \ref{lemma:relations-resolvent}), each term in the previous sum is non-positive. Since the connection is assumed not to be opaque, there is a non-trivial $u \in \ker(\X)$ (we assume it is $L^2$-normalized) which is pointwise orthogonal to $\mathbbm{1}_{\mc{E}}$ (see Lemma \ref{lemma:resonances-zero}), namely it is trace-free everywhere on $SM$. Note that, without loss of generality, we can always assume that $u$ is skew-Hermitian (see \eqref{eq:skewHermitiandiscussion}. 
Moreover, there exists $c > 0$ such that for $s$ small enough:
\[
\Re(\lambda_s) \leq - c \langle \Pi P_A u, P_A u \rangle_{L^2}.
\]
The question is therefore whether one can find $A \in C^\infty(M,T^*M\otimes \mathrm{End}_{\mathrm{sk}}(\E))$ such that $\langle \Pi P_A u, P_A u \rangle_{L^2} > 0$ and by Lemma \ref{lemma:relations-resolvent}, it is sufficient to produce such an $A$ such that $\Pi P_A u \neq 0$. If this is the case, the perturbation allows to eject \emph{at least one} resonance outside $0$ (since the resonances can only move in the half-plane $\left\{\Re(z) \leq 0\right\}$) and thus, iterating the process (there is only a finite number of resonances at $0$), one ejects all the resonances of $\X = (\pi^* \nabla^{\mathrm{End}(\E)})_X$ at $0$, except the one induced by $\mathbbm{1}_{\mc{E}}$. (This is the same argument as in \S\ref{ssection:absence-ckt} and \S\ref{sssection:absence-ckts-endomorphism}.)

\begin{lemma}
\label{lemma:perturbation}
There exists a perturbation $A \in C^\infty(M,T^*M\otimes \mathrm{End}_{\mathrm{sk}}(\E))$, arbitrarily small in any $C^k$-norm (for $k \geq 0$), such that one has $\langle \Pi P_A u, P_A u \rangle_{L^2} > 0$.
\end{lemma}

\begin{proof} We begin the proof with a preliminary geometric discussion. Consider a point $(x_0,v_0) \in SM$. Since $u(x_0,v_0) \neq 0$ (by the first item of Lemma \ref{lemma:resonances-zero}) and $\Tr(u(x_0,v_0)) = 0$, we can find $A_* \in \mathrm{End}_{\mathrm{sk}}(\E_{x_0})$ such that $P_{A_*}u (x_0,v_0) = [A_*,u(x_0,v_0)] \neq 0$. (Indeed, if not then $[A_*,u(x_0,v_0)]= 0$ for any skew-Hermitian $A_*$, thus for any Hermitian $A_*$, thus for any endomorphism $A_*$, and so $u(x_0,v_0)$ is a multiple of the identity but $u(x_0,v_0)$ is trace-free and non-zero, which is impossible.) We then define $A_1(x_0) := A_* \otimes v_0^* \in T_{x_0}^*M \otimes \mathrm{End}_{\mathrm{sk}}(\E_{x_0})$. If $f := [\pi_1^* A_1(x_0,\cdot), u(x_0,\cdot)] \in C^\infty(S_{x_0}M,\mathrm{End}_{\mathrm{sk}}(\E_{x_0}))$, then $f$ is not identically zero since $f(v_0)=[\pi_1^* A_1(x_0,v_0), u(x_0,v_0)] = [A_*,u(x_0,v_0)] \neq 0$.

Now, consider a perpendicular direction $w_0 \in S_{x_0}M$ to $v_0$ and define the $(n-2)$-dimensional sphere
\[
\Ss_{w_0}^{n-2} := S_{x_0}M \cap \left\{\langle v , w_0 \rangle=0 \right\},
\]
(in particular, $v_0 \in \Ss_{w_0}^{n-2}$). Consider the restriction of $f$ to $\Ss_{w_0}^{n-2}$ (still denoted by $f$). By construction, this is not identically $0$ and we can decompose it in Fourier series $f = \sum_{m \geq 0} f_m$, where each $f_m$ is a spherical harmonic of degree $m$ on $\Ss_{w_0}^{n-2}$. Since $f \neq 0$, there exists $m_0 \in \N$ such that $f_{m_0} \neq 0$. 

We then consider the operator
\[
Q_{m_0} : C^{\infty}(M,T^*M \otimes \mathrm{End}_{\mathrm{sk}}(\E)) \rightarrow C^\infty(M,\otimes^{m_0}_S T^*M \otimes \mathrm{End}(\E)),
\]
defined by
\[
Q_{m_0}(A) := - {\pi_{m_0}}_* \Pi P_u \pi_1^* A = {\pi_{m_0}}_* \Pi [\pi_1^* A, u],
\]
where $P_u f = [u,f]$, $f \in C^\infty(SM,\mathrm{End}(\E))$, is defined similarly to $P_A$.


Following the same arguments as in \cite[Theorem 3.1]{Guillarmou-17-1} (see also Remark 3.9 in the same paper), one proves that $Q_{m_0}$ is a pseudodifferential operator of order $-1$ (it is important here to use that $u$ is smooth). Following \cite[Theorem 4.4]{Gouezel-Lefeuvre-19}, one can compute its principal symbol and, given $(x_0,\xi_0) \in T^*M$ and $S \in C^\infty(M)$ such that $S(x_0)=0, \dd S(x_0) = \xi_0$, one finds that for any $A_{m_0} \in \otimes_S^{m_0} T^*_{x_0}M \otimes \End \E_{x_0}$: 
\begin{multline}
\label{equation:symbole-principal}
\langle \sigma(Q_{m_0})(x_0,\xi_0)A_1(x_0), A_{m_0}(x_0) \rangle_{x_0}  = \lim_{h \rightarrow 0} h^{-1} \langle Q_{m_0}(e^{i/h S} A_1), A_{m_0} \rangle_{x_0} \\
 = \dfrac{2\pi}{|\xi_0|} \int_{S_{x_0}M \cap \langle \xi_0, v \rangle = 0} \langle [\pi_1^* A_1,u](x_0,v), \pi_{m_0}^* A_{m_0}(x_0,v) \rangle_{x_0} \dd S_{\xi_0}(v),
\end{multline}
where $\dd S_{\xi_0}(v)$ is the canonical measure induced on the $(n-2)$-dimensional sphere $S_{x_0}M \cap \left\{\langle \xi_0, v \rangle = 0\right\} =: \Ss^{n-2}_{\xi_0^\sharp}$. It is now sufficient to show that the principal symbol of $Q_{m_0}$ is not zero in order to conclude: indeed, if $\langle{\Pi P_A u, P_Au}\rangle_{L^2} = 0$, by Lemma \ref{lemma:relations-resolvent} we have $\Pi P_A u = 0$, so if $\Pi P_A u = 0$ for all $A$, in particular we have $Q_{m_0} A = 0$ for all $A$, which cannot be the case if $\sigma(Q_{m_0}) \neq 0$. For the former claim, we consider the $v_0 \in S_{x_0}M$ introduced in the preceding discussion and $\xi_0 \in T^*_{x_0}M$ such that $\xi_0^\sharp \bot v_0$ (thus $v_0 \in \left\{\langle \xi_0,v\rangle=0\right\}$). We take $A_1(x_0) = A_* \otimes v_0^*$. We consider $A_{m_0} \in \otimes^{m_0}_S T^*_{x_0}M \otimes \mathrm{End}_{\mathrm{sk}}(\E_{x_0})$ such that the restriction of $\pi_{m_0}^* A_{m_0}$ to $\Ss^{n-2}_{\xi_0^\sharp}$ is equal to $f_{m_0}$. (For that, consider $A_{m_0} \in  \otimes^{m_0}_S (T^*_{x_0}M \cap (\R \xi_0)^\bot)|_{0-\Tr} \otimes \mathrm{End}_{\mathrm{sk}}(\E_{x_0})$ such that $\pi_{m_0}^* A_{m_0} = f_{m_0}$, which is always possible since
\[
\pi_{m_0}^* :  \otimes^{m_0}_S (T^*_{x_0}M \cap (\R \xi_0)^\bot)|_{0-\Tr} \otimes \mathrm{End}_{\mathrm{sk}}(\E_{x_0}) \rightarrow \Omega_{m_0}(\Ss^{n-2}_{\xi_0^\sharp}) \otimes \mathrm{End}_{\mathrm{sk}}(\E_{x_0})
\]
is an isomorphism, and then extend naturally $A_{m_0}$ as an element of $\otimes^{m_0}_S T^*_{x_0}M \otimes \mathrm{End}_{\mathrm{sk}}(\E_{x_0})$.) We then obtain using \eqref{equation:symbole-principal}:
\[
\begin{split}
\langle \sigma(Q_{m_0})(x_0,\xi_0)A_1(x_0), A_{m_0}(x_0) \rangle_{x_0} & = \dfrac{2\pi}{|\xi_0|} \int_{S_{x_0}M \cap \langle \xi_0, v \rangle = 0} \langle f(x_0,v), \pi_{m_0}^* A_{m_0}(x_0,v) \rangle_{x_0} \dd S_{\xi_0}(v)  \\
& = \dfrac{2\pi}{|\xi_0|} \int_{\Ss^{n-2}_{\xi_0^\sharp}} \sum_{m \geq 0} \langle f_m(x_0,v), f_{m_0}(x_0,v)  \rangle_{x_0}\dd S_{\xi_0}(v) \\
& = \dfrac{2\pi}{|\xi_0|} \|f_{m_0}\|^2_{L^2(\Ss^{n-2}_{\xi_0^\sharp})} > 0,
\end{split}
\]
since $f_{m_0} \neq 0$.
\end{proof}

This concludes the proof of Theorem \ref{theorem:generic-opacity}. 

\end{proof}

\bibliographystyle{alpha}
\bibliography{Biblio}

\end{document}